\documentclass[11pt]{amsart}
\usepackage{amsmath,amssymb,latexsym,cite}
\usepackage[small]{caption}
\usepackage{graphicx,color,mathrsfs}
\usepackage{subfigure,color}
\usepackage{cite}
\usepackage[colorlinks=true,urlcolor=blue,
citecolor=red,linkcolor=blue,linktocpage,pdfpagelabels,
bookmarksnumbered,bookmarksopen]{hyperref}
\usepackage[italian,english]{babel}
\usepackage{units}
\usepackage{enumitem}
\usepackage[left=2.7cm,right=2.7cm,top=2.8cm,bottom=2.8cm]{geometry}
\usepackage[hyperpageref]{backref}

\usepackage[colorinlistoftodos]{todonotes}

\numberwithin{equation}{section}
\newtheorem{theorem}{Theorem}[section]
\newtheorem{proposition}[theorem]{Proposition}
\newtheorem{lemma}[theorem]{Lemma}
\newtheorem{remark}[theorem]{Remark}

\newtheorem{corollary}[theorem]{Corollary}
\newtheorem{definition}[theorem]{Definition}
\theoremstyle{definition}

\renewcommand{\epsilon}{\eps}
\renewcommand{\i}{{\rm i}}

\newcommand{\N}{{\mathbb N}}
\newcommand{\R}{{\mathbb R}}

\newcommand{\eps}{\varepsilon}

\newcommand{\iu}{{\rm i}}

\newcommand{\pnorm}[2][]{\if #1'' \left|#2\right|_p \else \left|#2\right|_{#1} \fi}

\newcommand{\C}{\mathbb{C}}

\newcommand{\Ss}{\mathscr{S}}

\renewcommand{\theta}{\vartheta}

\title[Ground states for fractional magnetic operators]{Ground states for fractional magnetic operators}

\author[P.\ d'Avenia]{Pietro d'Avenia}

\address[P.\ d'Avenia]{Dipartimento di Meccanica, Matematica e Management
\newline\indent 
Politecnico di Bari,
Via Orabona 4,  I-70125  Bari, Italy}
\email{\href{mailto:pietro.davenia@poliba.it}{pietro.davenia@poliba.it}}

\author[M.\ Squassina]{Marco Squassina}

\address[M.\ Squassina]{Dipartimento di Matematica e Fisica \newline\indent
	Universit\`a Cattolica del Sacro Cuore,
	Via dei Musei 41, I-25121 Brescia, Italy}
\email{\href{marco.squassina@unicatt.it}{marco.squassina@unicatt.it}}



\begin{document}

\begin{abstract}
We study a class of minimization problems for a nonlocal operator
involving an external magnetic potential. The notions are physically justified
and consistent with the case of absence of magnetic fields. Existence of solutions
is obtained via concentration compactness.
\end{abstract}

\maketitle

\begin{center}
	\begin{minipage}{9cm}
		\small
		\tableofcontents
	\end{minipage}
\end{center}
\medskip

\section{Introduction and results}
Since the late nineties, nonlocal integral operators like
\begin{equation}
\label{slaplacian}
(-\Delta)^s u (x) = c_s\lim_{\eps\searrow 0}\int_{B^c_\eps(x)}\frac{u(x)-u(y)}{|x-y|^{3+2s}} dy=\mathcal{F}^{-1}(|\xi|^{2s}\mathcal{F}(u)(\xi))(x), \qquad u\in C^\infty_c(\R^3), 
\end{equation}
where $s\in(0,1)$ and 
$$
c_{s} = s 2^{2s} \frac{\Gamma\big(\frac{3+2s}{2}\big)}{\pi^{3/2}\Gamma(1-s)},
$$
being $\Gamma$ the Gamma function, have been widely used 
in the theory of L\'evy processes. Indeed, in view of  
the L\'evy-Khintchine formula, the 
generator ${\mathscr H}$ of the semigroup on $C^\infty_c(\R^3)$ associated to a general L\'evy process is given by 
\begin{equation}
\label{LK}
{\mathscr H}u(x)= -a_{ij}\partial_{x_i x_j}^2 u(x)-b_i\partial_{x_i}u(x)-\lim_{\eps\searrow 0}\int_{B^c_\eps(0)}\Big(u(x+y)-u(x)-1_{\{|y|<1\}}(y)y\cdot\nabla u(x)\Big)d\mu,
\end{equation}
with summation on repeated indexes and where $\mu$ is a L\'evy nonnegative measure, namely 
$$
\int_{\R^3}\frac{|y|^2}{1+|y|^2}d\mu<\infty.
$$
The last contribution in \eqref{LK}
represents the purely jump part of the {\em L\'evy process}, while the first two terms represent a {\em Brownian motion} with drift. 
It is now well established that L\'evy processes with jumps are more appropriate for some mathematical models in finance.
Among L\'evy processes, the only stochastically stable ones having jump part are those corresponding to radial measures as
$$
d\mu=\frac{c_s}{|y|^{3+2s}}dy,
$$
hence the importance of the definition \eqref{slaplacian}. Moreover, the fractional Laplacian \eqref{slaplacian} allows to develop a generalization of quantum mechanics and also to describe the motion of a chain or array of particles that are connected by elastic springs and unusual diffusion processes 
in turbulent fluid motions and material transports in fractured media  (for more details see e.g. \cite{A,CT,Lask3,MK} and the references therein).
Due to the results of Bourgain-Br\'ezis-Mironescu \cite{bourg,bourg2}, up to correcting the operator \eqref{slaplacian} with the
factor $(1-s)$ it follows that $(-\Delta)^s u$ converges to $-\Delta u$ in the limit $s\nearrow 1$.\ Thus, up to normalization, we may 
think the nonlocal case as an approximation of the local case.\\
A {\em pseudorelativistic extension} of the Laplacian is the well known pseudodifferential operator $\sqrt{-\Delta + m^2} - m$ where $m$ is a nonnegative number.
This operator appears in the study of free relativistic particles of mass $m$ and $\sqrt{-\Delta + m^2}$ is defined by  
$\mathcal{F}^{-1}(\sqrt{|\xi|^2+ m^2}\mathcal{F}(u)(\xi))$ (see \cite{LS} for more details). 
We observe that for $m=0$ we have the operator in \eqref{slaplacian} with $s=1/2$.\\
An important role in the study of particles which interact, e.g. using the Weyl covariant derivative, with a magnetic field $B=\nabla\times A$, $A:\R^3\to\R^3$,  is assumed by another {\em extension} of the Laplacian, namely the {\em magnetic Laplacian} $(\nabla-\iu A)^2$ (see \cite{AHS,ReedSimon}).
Nonlinear magnetic Schr\"odinger equations like
\[
- (\nabla-\iu A)^2 u + u = f(u)
\]
have been extensively studied (see e.g.\ \cite{arioliSz,CS,DCVS,EL,K,S}).\\
In \cite{IT}, Ichinose and Tamura, through oscillatory integrals, introduce the so-called Weyl pseudodifferential operator defined with {\em mid-point prescription}
\begin{align*}
{\mathscr H}_A u (x)
&=
\frac{1}{(2\pi)^3}\int_{\R^6} e^{\i (x-y)\cdot \xi } \sqrt{\Big|\xi - A\big(\frac{x+y}{2}\big)\Big|^2 + m^2} u(y) dyd\xi\\
&=
\frac{1}{(2\pi)^3}\int_{\R^6} e^{\i (x-y)\cdot \left(\xi+ A\big(\frac{x+y}{2}\big)\right) } \sqrt{|\xi|^2 + m^2} u(y) dyd\xi
\end{align*}
as a {\em fractional relativistic} generalization of the magnetic Laplacian (see also \cite{I89}, the review article \cite{I10} and the references therein).\
The operator ${\mathscr H}_A$ takes the place of $\sqrt{-\Delta + m^2}$ and it is possible to show that for all $u\in C^\infty_c(\R^3,\C)$,
\begin{equation*}
\begin{split}
{\mathscr H}_A u (x)
&=
m u(x)
- \lim_{\eps\searrow 0}\int_{B^c_\eps(0)} \left[e^{-\i y \cdot A \big( x+ \frac{y}{2}\big)} u(x+y) - u(x) - 1_{\{|y|<1\}}(y) y\cdot (\nabla - \i A(x))u(x)\right]d\mu\\
&=
m u(x)
+ \lim_{\eps\searrow 0}\int_{B^c_\eps(x)} \left[u(x)-e^{\i (x-y)\cdot A\left(\frac{x+y}{2}\right)}u(y)\right] \mu(y-x)dy ,
\end{split}
\end{equation*}
where
\[
d\mu=
\mu(y)dy = 
\begin{cases}
2\left(\frac{m}{2\pi}\right)^2 \frac{K_2(m|y|)}{|y|^2}dy,
& m>0,\\
\frac{1}{\pi^2 |y|^4}dy, & m=0,
\end{cases}
\]
and $K_2$ is the modified Bessel function of the third kind of order $2$ (see e.g. \cite[Subsection 3.1]{I10}).\\
In this paper we are concerned with the operator
\begin{equation}
	\label{operator}
(-\Delta)^s_Au(x)=c_s \lim_{\eps\searrow 0}\int_{B^c_\eps(x)}\frac{u(x)-e^{\i (x-y)\cdot A\left(\frac{x+y}{2}\right)}u(y)}{|x-y|^{3+2s}}dy,
\quad x\in\R^3,
\end{equation}
and, in particular, with ground state solutions of the equation
\begin{equation}
\label{prob}
\tag{$\mathscr{P}_{s,A}$}
(-\Delta)^s_A u+u=|u|^{p-2}u\quad \text{in $\R^3$.}
\end{equation}
The operator \eqref{operator} is consistent with the definition of fractional Laplacian given in  \eqref{slaplacian} if $A=0$ and with 
${\mathscr H}_A$ for $m=0$ and $s=1/2$. To our knowledge, this is the first mathematical contribution
to the study of nonlinear problems involving operator~\eqref{operator}.\ 

For the sake of completeness we mention that there exist other
different definitions of the {\em magnetic pseudorelativistic
	operator} (see \cite{I10,IMP,LS}) and in \cite{franketal} a fractional
magnetic operator $(\nabla-\iu A)^{2s}$ is defined through the
spectral theorem (see also discussion on the different definitions in
\cite[Proposition 2.6]{I10}).
\vskip4pt
\noindent
Throughout the paper we consider magnetic potentials $A$'s which have locally bounded gradient. We now state our results.
\\
Let $2<p<6/(3-2s)$ and 
consider the minimization problem
\begin{equation}
\label{MA}
\tag{${\mathscr M}_A$}
{\mathscr M}_A=\inf_{u\in \Ss} \left(\int_{\R^3}|u|^2dx+\frac{c_s}{2}\int_{\R^6}\frac{|e^{-\i (x-y)\cdot A\left(\frac{x+y}{2}\right)}u(x)-u(y)|^2}{|x-y|^{3+2s}}dxdy\right),
\end{equation}
where
\begin{equation*}
\Ss=\Big\{u\in H^s_A(\R^3,\C):\int_{\R^3}|u|^p dx=1\Big\}
\end{equation*}
and $H^s_A(\R^3,\C)$ is a suitable Hilbert space defined in Section~\ref{setting}.
Once a solution to ${\mathscr M}_A$ exists, due to the Lagrange Multiplier Theorem, we get a weak solution to \eqref{prob}, see Sections \ref{setting} and \ref{sect4}.\\
When $\Ss$ is restricted to radially symmetric functions, the problem is denoted by ${\mathscr M}_{A,r}$.

\vskip3pt
First we give the following
\begin{definition}\rm
	We say that $A$ satisfies assumption $\mathscr{A}$, if 
	for any unbounded sequence $\Xi=\{\xi_n\}_{n\in\N}\subset\R^3$
	there exist a sequence $\{H_n\}_{n\in\N}\subset\R^3$ and a function $A_\Xi:\R^3\to\R^3$ such that 
	\begin{equation}
	\label{limA}
	\lim_n A_n(x)=A_\Xi(x) \hbox{ for all } x\in\R^3 \hbox{ and }
	\sup_n \|A_n\|_{L^\infty(K)}<\infty  \hbox{ for all compact sets } K,
	\end{equation}
	where $A_n(x):=A(x+\xi_n)+H_n$ and $\{\xi_n\}$ is a subsequence of $\Xi$ such that $|\xi_n|\to\infty$. 
\end{definition}
\noindent
We also set $\mathscr{X}:=\{\Xi=\{\xi_n\}_{n\in\N}\hbox{ unbounded}: \text{condition \eqref{limA} holds}\}$.
Observe that, if $A$ admits limit as $|x|\to\infty$,
then it satisfies assumption $\mathscr{A}$.\\
Our main result is
\begin{theorem}[Subcritical case]
	\label{main}
	The following facts hold:
	\begin{enumerate}[label=(\roman*),ref=\roman*]
	\item \label{i12} ${\mathscr M}_{A,r}$  has a solution; 
	\item \label{iii12}if $A$ is linear, then ${\mathscr M}_A$ has a solution;
	\item \label{ii12}if $A$ satisfies $\mathscr{A}$ and 
	${\mathscr M}_A<\inf_{\Xi\in \mathscr{X}} {\mathscr M}_{A_\Xi}$, 
	then ${\mathscr M}_A$ has a solution.
	\end{enumerate}
\end{theorem}

\noindent
We also consider the minimization problem 
\begin{equation}
\label{MAC}
\tag{${\mathscr M}_A^c$}
{\mathscr M}_A^c:=\inf_{u\in \Ss^c} 
\frac{c_s}{2}\int_{\R^6}\frac{|e^{-\i (x-y)\cdot A\left(\frac{x+y}{2}\right)}u(x)-u(y)|^2}{|x-y|^{3+2s}}dxdy,
\end{equation}
where 
$$
\Ss^c=\Big\{u\in D^s_A(\R^3,\C):\int_{\R^3}|u|^{6/(3-2s)} dx=1\Big\}
$$
and $D^s_A(\R^3,\C)$ is a suitable Hilbert space defined in Subection~\ref{subscrit}.
We are able to prove
\begin{theorem}[Critical case]
	\label{main-2}
	The following facts hold:
	\begin{enumerate}[label=(\roman*),ref=\roman*]
		\item \label{i13}if ${\mathscr M}_A^c$ has a solution $u$, there exist
		$z\in\R^3$, $\eps>0$ and $\vartheta_A:\R^3\to\R$ such that
		$$
		u(x)=d_{s}\left(\frac{\eps}{\eps^2+|x-z|^2}\right)^{\frac{3-2s}{2}}e^{\i \vartheta_A(x)};
		$$ 
		\item \label{ii13}if for some $k\in\N$ and $E\subset\R^6$ of positive measure 
		$$
		(x-y)\cdot A\Big(\frac{x+y}{2}\Big)\not \equiv 2k\pi\quad \hbox{for all } (x,y)\in E,
		$$
		then ${\mathscr M}_A^c$ has no solution
		$u$ of the form $e^{\i \vartheta} v(x)$ where $\theta\in\R$ and $v$ of fixed sign.
	\end{enumerate}
\end{theorem}
\vskip3pt
\noindent
The local version of the above results can be found in the work \cite{EL} by Esteban and Lions.
In \cite{DipValPal}, for the case without magnetic field and with subcritical nonlinearities, 
existence of ground states was obtained using different arguments, 
namely without involving concentration compactness arguments, but instead symmetrizing the minimizing sequences, by using
$$
\int_{\R^6}\frac{||u(x)|^*-|u(y)|^*|^2}{|x-y|^{3+2s}}dxdy\leq \int_{\R^6}\frac{|u(x)-u(y)|^2}{|x-y|^{3+2s}}dxdy,
$$
for all $u\in H^s(\R^3)$, where $v^*$ denotes the Schwarz
symmetrization of $v:\R^3\to \R^+$. On the contrary, when $A\not\equiv 0$, the inequality 
$$
\int_{\R^6}\frac{|e^{-\i (x-y)\cdot A\left(\frac{x+y}{2}\right)}|u(x)|^*-|u(y)|^*|^2}{|x-y|^{3+2s}}dxdy\leq \int_{\R^6}\frac{|e^{-\i (x-y)\cdot A\left(\frac{x+y}{2}\right)}u(x)-u(y)|^2}{|x-y|^{3+2s}}dxdy,
$$
does not seem to work and a different strategy for the proof has to be outlined.\ 
Dealing with the nonlocal case, it is natural to expect that, in the study of minimizing sequences, the hardest stage
is that of ruling out the dichotomy in the concentration compactness alternative. This is in fact the case, but thanks to a careful analysis 
developed in Lemma~\ref{lemma-dic}, dichotomy can be ruled out allowing for tightness and hence the strong convergence of
minimizing sequences up translations and phase changes.

We organize the paper in the following way: in Section \ref{setting} we introduce the functional setting of the problem and we provide some basic properties about it; in Section \ref{sec3} we show further technical facts on the functional setting as well as some preliminary results about the Concentration-Compactness procedure; finally, in Section \ref{sec4}, we complete with the proofs of our results.

\vskip10pt
\noindent
{\bf Acknowledgments.} The research  was partially supported by Gruppo Nazionale per l'Analisi Matematica,
la Probabilit\`a e le loro Applicazioni (INdAM).

\vskip10pt
\noindent
{\bf Notations.} We denote by $B_R(\xi)$ a ball in $\R^3$ of center $\xi$ and radius $R$.
For a measurable set $E\subset \R^3$ we denote by $E^c$ the complement
of $E$ in $\R^3$, namely $E^c=\R^3\setminus E$.\ We denote by $1_E$ the indicator function of $E$.\
The symbol ${\mathcal L}^n(\Omega)$ stands for the Lebesgue measure of a measurable subset $\Omega\subset\R^n$. For a 
complex number $z\in\C$, the symbol $\Re z$ indicates its real part and $\Im z$ its imaginary part.
The modulus of $z$ is denoted by $|z|$. The standard norm of $L^p$ spaces is denoted by $\|\cdot\|_{L^p}.$

\section{Functional setting}
\label{setting}
\noindent
Let $L^2(\R^3,\C)$ be the Lebesgue space of complex valued functions with 
summable square endowed with the real scalar product
$$
\langle u,v \rangle_{L^2}:=\Re\int_{\R^3} u\bar vdx,\quad \text{for all $u,v\in L^2(\R^3,\C)$},
$$
and $A:\R^3\to\R^3$ be a continuous function. We consider the magnetic
Gagliardo semi-norm defined by
$$
[u]^2_{s,A}:=\frac{c_s}{2}\int_{\R^6}\frac{|e^{-\i (x-y)\cdot A\left(\frac{x+y}{2}\right)}u(x)-u(y)|^2}{|x-y|^{3+2s}}dxdy,
$$
the scalar product defined by
$$
\langle u,v \rangle_{s,A}:=\langle u,v \rangle_{L^2}+\frac{c_s}{2}{\Re}\int_{\R^6}\frac{\left(e^{-\i (x-y)\cdot A\left(\frac{x+y}{2}\right)}u(x)-u(y)\right)\overline{\left(e^{-\i (x-y)\cdot A\left(\frac{x+y}{2}\right)}v(x)-v(y)\right)}}{|x-y|^{3+2s}}dxdy,
$$
and the corresponding norm denoted by
$$
\|u\|_{s,A}:=\big(\|u\|_{L^2}^2+[u]_{s,A}^2\big)^{1/2}.
$$
We consider the space ${\mathcal H}$ of measurable functions $u:\R^3\to\C$ such that $\|u\|_{s,A}<\infty.$

\begin{proposition}
	\label{complete}
$({\mathcal H},\langle \cdot,\cdot \rangle_{s,A})$ is a real Hilbert space.
\end{proposition}
\begin{proof}
It is readily checked that $\langle u,v \rangle_{s,A}$ is a real scalar product. Let us prove that ${\mathcal H}$ with this scalar product is complete.
Let $\{u_n\}_{n\in\N}$ be a Cauchy sequence in ${\mathcal H}$, namely for every $\varepsilon >0$ there exists $\nu_\varepsilon\in\mathbb{N}$ such that for all $m,n>\nu_\epsilon$ we have $\|u_n - u_m\|_{s,A}<\varepsilon$. Thus $\{u_n\}_{n\in\N}$ is a Cauchy sequence on $L^2(\R^3,\C)$ and then there exists $u\in L^2(\R^3,\C)$ such that $u_n \to u$ in $L^2(\R^3,\C)$ and a.e.\ in $\R^3$. Firstly, we prove that $u\in {\mathcal H}$. By Fatou Lemma we have
\[
[u]^2_{s,A}
\leq
\liminf_n
[u_n]^2_{s,A}
\leq
\liminf_n
([u_n - u_{\nu_1}]_{s,A} + [u_{\nu_1}]_{s,A})^2
\leq
(1+ [u_{\nu_1}]_{s,A})^2.
\]
Thus it remains to prove that $[u_n - u]_{s,A}\to 0$ as $n\to\infty$. Again, by Fatou Lemma
\[
[u_n - u]_{s,A} \leq \liminf_k [u_n - u_k]_{s,A} \leq \liminf_k \|u_n - u_k\|_{s,A} \leq \eps,
\]
for all $\varepsilon >0$ and $n$ large.
\end{proof}
\noindent
For any function $w:\R^3\to\C$ and a.e.\ $x\in\R^3$, we set
\begin{equation}
\label{notation}
w_x(y):=e^{\i (x-y)\cdot A\left(\frac{x+y}{2}\right)}w(y),\quad \text{for $y\in\R^3$}.
\end{equation}
We have
\begin{proposition}
The space $C^\infty_c(\R^3,\C)$ is a subspace of ${\mathcal H}$.	
\end{proposition}
\begin{proof}
	It is enough to prove that $[u]_{s,A}<\infty$, for any $u\in C^\infty_c(\R^3,\C)$.\ If $K$ is the compact support of $u$, we have
	\begin{align*}
	\int_{\R^6}\frac{|e^{-\i (x-y)\cdot A\left(\frac{x+y}{2}\right)}u(x)-u(y)|^2}{|x-y|^{3+2s}}dxdy\leq 
	2	\int_{K\times\R^3}\frac{|u_x(x)-u_x(y)|^2}{|x-y|^{3+2s}}dxdy.
	\end{align*}
Observe that, since $\nabla A$ is locally bounded, the gradient of the function $(x,y)\mapsto u_x(y)$ is bounded on $K\times\R^3$. 
Then we have $|u_x(x)-u_x(y)|\leq C|x-y|$ for any $(x,y)\in K\times\R^3$. Of course, we also have 
$|u_x(x)-u_x(y)|\leq C$ for any $(x,y)\in K\times\R^3$. Hence, we get
\begin{align*}
\int_{K\times\R^3}\frac{|u_x(x)-u_x(y)|^2}{|x-y|^{3+2s}}dxdy&\leq C\int_{K\times\R^3}\frac{\min\{|x-y|^2,1\}}{|x-y|^{3+2s}}dxdy \\
&\leq C\int_{B_1(0)} \frac{1}{|z|^{1+2s}} dz+C\int_{B^c_1(0)}  \frac{1}{|z|^{3+2s}}dz,
\end{align*}
which concludes the proof.
\end{proof}
\noindent
Thus we can give the following
\begin{definition}
We define $H^s_A(\R^3,\C)$ as the closure of $C^\infty_c(\R^3,\C)$ in ${\mathcal H}$.	
\end{definition}
\noindent
Then, $H^s_A(\R^3,\C)$ is a real Hilbert space by Proposition~\ref{complete}.\
For $A=0$ this space is consistent with the usual fractional space $H^s(\R^3,\C)$ whose norm is denoted by $\|\cdot\|_{s}$.
For a given Lebesgue measurable set $E\subset\R^3$ 
the localized Gagliardo norms are defined by
\begin{align*}
\|u\|_{H^s_A(E)}&:=\left(\int_{E} |u(x)|^2dx+
\frac{c_s}{2}\int_{E\times E}\frac{|e^{-\i (x-y)\cdot A\left(\frac{x+y}{2}\right)}u(x)-u(y)|^2}{|x-y|^{3+2s}}dxdy\right)^{1/2}, \\
\|u\|_{H^s(E)}&:=\left(\int_{E} |u(x)|^2dx+
\frac{c_s}{2}\int_{E\times E}\frac{|u(x)-u(y)|^2}{|x-y|^{3+2s}}dxdy\right)^{1/2}.
\end{align*}
The operator $(-\Delta)^s_A : H^s_A(\R^3,\C)\to H^{-s}_A(\R^3,\C)$ is defined by duality as
\begin{align*}
\langle (-\Delta)^s_Au,v\rangle &:=\frac{c_s}{2} {\Re}\int_{\R^6}\frac{\left(e^{-\i (x-y)\cdot A\left(\frac{x+y}{2}\right)}u(x)-u(y)\right)\overline{\left(e^{-\i (x-y)\cdot A\left(\frac{x+y}{2}\right)}v(x)-v(y)\right)}}{|x-y|^{3+2s}}dxdy  \\
& =\frac{c_s}{2} {\Re}\int_{\R^6}\frac{\left(u(x)-e^{\i (x-y)\cdot A\left(\frac{x+y}{2}\right)}u(y)\right)\overline{\left(v(x)-e^{\i (x-y)\cdot A\left(\frac{x+y}{2}\right)}v(y)\right)}}{|x-y|^{3+2s}}dxdy.
\end{align*}
If $f\in H^{-s}_A(\R^3,\C)$, we say that $u\in H^s_A(\R^3,\C)$ is a weak solution to 
\begin{equation}
\label{problema}
(-\Delta)^s_Au+u=f,\quad\text{in $\R^3$,}
\end{equation}
if we have
\begin{align*}
&\frac{c_s}{2}{\Re}\int_{\R^6}\frac{\left(u(x)-e^{\i (x-y)\cdot A\left(\frac{x+y}{2}\right)}u(y)\right)\overline{\left(v(x)-e^{\i (x-y)\cdot A\left(\frac{x+y}{2}\right)}v(y)\right)}}{|x-y|^{3+2s}}dxdy  \\
&+{\Re}\int_{\R^3} u\overline{v} dx={\Re}\int_{\R^3} f\overline{v} dx,\quad \text{for all $v\in H^s_A(\R^3,\C)$}.
\end{align*}
Of course,  one can equivalently define the weak solution 
by testing over functions $v\in C^\infty_c(\R^3,\C)$.\\
On smooth functions, the operator $(-\Delta)^s_A $ admits the point-wise representation \eqref{operator}. To show this we need the following preliminary results.
\begin{lemma}
	\label{prelimll}
Let $K$ be a compact subset of $\R^3,$ $R>0$ and set $K'=\{x\in\R^3: d(x,K)\leq R\}$.
Assume that $f\in C^2(\R^6)$
and that $g\in C^{1,\gamma}(K')$
for some $\gamma\in [0,1]$.\
If $h(x,y)=f(x,y)g(y)$, then there exists a positive constant $C$ depending on $K,f,g,R$, such that
$$
|\nabla_y h(x,y_2)-\nabla_y h(x,y_1)|\leq C|y_2-y_1|^\gamma,
$$
for all $x\in K$ and every $y_2,y_1\in K'$.
\end{lemma}
\begin{proof}
The proof is omitted as it is straightforward.
\end{proof}


\begin{lemma}
	\label{controllo}
Let $A\in C^2(\R^3)$ and $u\in C^{1,\gamma}_{\rm loc}(\R^3,\C)$ for some $\gamma\in [0,1]$.\
Then, for any compact set $K\subset\R^3$ and $R>0$,
there exists a positive constant $C$ depending on $R,K,A,u$, such that
$$
|u_x(x+y)+u_x(x-y)-2u_x(x)|\leq C |y|^{1+\gamma},
$$
for every $x\in K$ and $y \in  B_R(0)$.
\end{lemma}
\begin{proof}
	Fix a compact set $K\subset \R^3$ and $R>0$.\ Consider $x\in K$ and $y\in B_R(0)$.\
	Then, by the Mean Value Theorem, 
	there exist $\tau_1,\tau_2\in [0,1]$ such that
\begin{align*}
|u_x(x+y)+u_x(x-y)-2u_x(x)|&=|\nabla_y u_x(x+\tau_1 y)\cdot y- \nabla_y u_x(x-\tau_2 y)\cdot y|\\
& \leq |\nabla_y u_x(x+\tau_1 y)- \nabla_y u_x(x-\tau_2 y)||y|\leq C|y|^{1+\gamma},
\end{align*}
where in the last inequality we use Lemma~\ref{prelimll} with $f(x,y)=e^{\i (x-y)\cdot A\left(\frac{x+y}{2}\right)}$ and $g(y)=u(y).$
\end{proof}

\noindent
Thus in the case $u$ and $A$ are smooth enough, we have the following result
\begin{theorem}[Weak to strong solution]
	\label{weakstrong}
	Let $u\in H^s_A(\R^3,\C)$ be a weak solution to \eqref{problema}. Assume that
	$A\in C^2(\R^3)$ and that
	$$
	u\in L^\infty(\R^3,\C)\cap C^{1,\gamma}_{\rm loc}(\R^3,\C),\quad\text{for some 
		$\gamma\in (0,1]$ with $\gamma>2s-1$.}
	$$
	Then $u$ solves problem \eqref{problema} pointwise a.e.\ in $\R^3$.
\end{theorem}
\begin{proof}
With the notation introduced in \eqref{notation}, the definition of weak solution writes as
\begin{equation}
\label{varf-2}
\frac{c_s}{2}{\Re}\int_{\R^6}\frac{\left(u_x(x)-u_x(y)\right)\overline{\left(v_x(x)-v_x(y)\right)}}{|x-y|^{3+2s}}dxdy  
+{\Re}\int_{\R^3} u\overline{v} dx={\Re}\int_{\R^3} f\overline{v} dx,
\end{equation}
for all $v\in C^\infty_c(\R^3,\C)$.\
Let us fix a $v\in C^\infty_c(\R^3,\C)$ and set $K:={\rm supp}(v)$.\
Now, for any $\eps>0$, we introduce the auxiliary function
$g_\eps:K\to\R$ defined by
$$
g_\eps(x):=\frac{c_s}{2}\int_{\R^3}\frac{u_x(x)-u_x(y)}{|x-y|^{3+2s}}1_{B^c_\eps(x)}(y)dy.
$$
Note that for all $x\in K$ we have that
\begin{equation}
\label{puntgeps}
g_\eps (x)\to \frac{1}{2}(-\Delta)^s_A u (x),\quad\text{as $\eps\to 0$ whenever the limit exists}.
\end{equation}
Simple changes of variables show that $g_\eps$ can be equivalently written as
$$
g_\eps(x)=-\frac{c_s}{4}\int_{\R^3}
\frac{u_x(x+y)+u_x(x-y)-2u_x(x)}{|y|^{3+2s}}1_{B^c_\eps(0)}(y)dy.
$$
Furthermore, by Lemma~\ref{controllo}, there exist $C>0$ and $R>0$ such that
$$
|u_x(x+y)+u_x(x-y)-2u_x(x)|\leq C |y|^{1+\gamma},\quad \text{for $x\in K$
and $y \in  B_R(0)$}.
$$
Therefore, taking into account that
$|u_x(y)|\leq \|u\|_{L^\infty}$ for all $y\in\R^3$, 
we have the inequality
$$
\frac{|u_x(x+y)+u_x(x-y)-2u_x(x)|}{|y|^{3+2s}}\leq \frac{C}{|y|^{2+2s-\gamma}}1_{B_R(0)}(y)+
\frac{C}{|y|^{3+2s}}1_{B^c_R(0)}(y),
$$
for some constant $C$.\
Due to the assumption $\gamma>2s-1$, the right hand side belongs to $L^1(\R^3)$.
Then, by dominated convergence,  the limit of $g_\eps(x)$ as $\eps\to 0$
exists a.e.\ in $K$ and it is thus equal to $\frac{1}{2}(-\Delta)^s_A u (x)$ by \eqref{puntgeps}. Since
also $|g_\eps(x)|\leq C$ a.e.\ in $K$, again the dominated convergence yields
\begin{equation}
\label{loc-conv}
g_\eps\to \frac{1}{2}(-\Delta)^s_A u,\quad\text{strongly in $L^1(K)$}.
\end{equation}
Now, the first term in formula \eqref{varf-2} can be treated as follows
\begin{align*}
& \frac{c_s}{2}\int_{\R^6}\frac{\left(u_x(x)-u_x(y)\right)
	\overline{\left(v_x(x)-v_x(y)\right)}}{|x-y|^{3+2s}}dxdy \\
& =\lim_{\eps\to 0}\frac{c_s}{2}\int_{\R^6} \frac{\left(u_x(x)-u_x(y)\right)
	\overline{\left(v_x(x)-v_x(y)\right)}}{|x-y|^{3+2s}}1_{B^c_\eps(x)}(y)dxdy \\
&  =\lim_{\eps\to 0}\Big(\int_{\R^3} g_\eps(x)\overline{v(x)}dx 
-\frac{c_s}{2}\int_{\R^6}\frac{\left(u_x(x)-u_x(y)\right)
	\overline{v_x(y)}}{|x-y|^{3+2s}}1_{B^c_\eps(x)}(y)dxdy\Big).
\end{align*}
By Fubini Theorem on the second term of the last equality, switching
the two variables and observing that
$$-\left(u_y(y)-u_y(x)\right)
	\overline{v_y(x)}1_{B^c_\eps(y)}(x)
=(u_x(x)-u_x(y))\overline{v(x)}1_{B^c_\eps(x)}(y)$$
yields
\begin{align*}
\frac{c_s}{2}\int_{\R^6}\frac{\left(u_x(x)-u_x(y)\right)
	\overline{\left(v_x(x)-v_x(y)\right)}}{|x-y|^{3+2s}}dxdy
=\lim_{\eps\to 0}\int_{\R^3} 2g_\eps(x)\overline{v(x)}dx=\int_{\R^3} (-\Delta)^s_A u(x)\overline{v(x)}dx,
\end{align*}
where we used \eqref{loc-conv} in the last equality. Then, from formula~\eqref{varf-2},
we conclude that
$$ 
\Re\Big(\int_{\R^3} \big((-\Delta)^s_A u+u-f\big)\overline{v} dx\Big)=0,
\quad\text{for all $v\in C^\infty_c(\R^3,\C)$,}
$$
yielding $(-\Delta)^s_A u+u=f$ a.e.\ in $\R^3$. The proof is complete.
\end{proof}

\noindent
We conclude the section with an observation about the formal consistency of the
spaces $H^s_A(\R^3,\C)$, up to suitably correcting the norm, with the usual local Sobolev spaces without magnetic field
in the singular limit as $s\to 1$ and $A\to 0$ pointwise. Consider the modified norm
	$$
	|||u|||_{s,A}:=\big(\|u\|_{L^2}^2+(1-s)[u]_{s,A}^2\big)^{1/2}.
	$$
	By arguing as in the proof of Lemma~\ref{livelli},	it follows that
	$$
	\lim_{A\to 0}[u]_{s,A}^2=[u]_{s,0}^2, \quad \text{for all $u\in C^\infty_c(\R^3,\C)$}. 
	$$
	Moreover, from the results of Brezis-Bourgain-Mironescu \cite{bourg,bourg2}, 
	we know that
	$$
	\lim_{s\to 1}(1-s)[u]_{s,0}^2=\|\nabla u\|_{L^2}^2, 
	\quad \text{for all $u\in C^\infty_c(\R^3,\C)$}. 
	$$
	In conclusion
	$$
	\lim_{s\to 1}
	\lim_{A\to 0}
	|||u|||_{s,A}=\|u\|_{H^1(\R^3)},
	\quad \text{for all $u\in C^\infty_c(\R^3,\C)$}. 
	$$
	Hence $|||u|||_{s,A}$ approximates the $H^1$-norm for $s\sim 1$ and $A\sim 0$.

\section{Preliminary stuff}\label{sec3}

In this section we provide some technical facts about the functional setting of the problem as well as some preliminary results about the Concentration-Compactness procedure.

\begin{lemma}[Diamagnetic inequality]
	\label{diamagnetic}
For every $u\in H^s_A(\R^3,\C)$ it holds $|u|\in H^s(\R^3)$. More precisely
$$
\| |u|\|_{s}\leq \|u\|_{s,A},\quad\text{for every  $u\in H^s_A(\R^3,\C)$.}
$$
\end{lemma}
\begin{proof}
For a.e.\ $x,y\in \R^3$ we have
$$
\Re \big(e^{-\i (x-y)\cdot A\left(\frac{x+y}{2}\right)} u(x)\overline{u(y)}\big)\leq |u(x)||u(y)|.
$$
Therefore, we have
\begin{align*}
|e^{-\i (x-y)\cdot A\left(\frac{x+y}{2}\right)}u(x)-u(y)|^2 &=|u(x)|^2+|u(y)|^2-2\Re \big(e^{-\i (x-y)\cdot A\left(\frac{x+y}{2}\right)} u(x)\overline{u(y)}\big) \\
& \geq |u(x)|^2+|u(y)|^2-2|u(x)||u(y)|=||u(x)|-|u(y)||^2,
\end{align*}
which immediately yields the assertion.	
\end{proof}

\begin{remark}[Pointwise Diamagnetic inequality]\label{diamremark}
\rm
There holds 
$$
||u(x)|-|u(y)||\leq
|e^{-\i (x-y)\cdot A\left(\frac{x+y}{2}\right)}u(x)-u(y)|,\quad\text{for a.e.\ $x,y\in\R^3$.}
$$
\end{remark}

\noindent
We have the following local embedding of $H_A^s(\R^3,\C)$.
\begin{lemma}[Local embedding in $H^s(\R^3,\C)$]
	\label{localemb}
	For every compact set $K\subset \R^3,$ the space $H^s_A(\R^3,\C)$
	is continuously embedded into $H^s(K,\C)$.
\end{lemma}
\begin{proof}
Fixed a compact $K\subset\R^3$, we have
\begin{align*}
\|u\|_{H^s(K)}^2 &=\int_{K} |u(x)|^2dx+
\frac{c_s}{2}\int_{K\times K}\frac{|u(x)-u(y)|^2}{|x-y|^{3+2s}}dxdy \\
& \leq \int_{\R^3} |u(x)|^2dx+
C\int_{K\times K}\frac{|e^{-\i (x-y)\cdot A\left(\frac{x+y}{2}\right)}u(x)-u(y)|^2}{|x-y|^{3+2s}}dxdy \\
&\quad +C\int_{K\times K}\frac{|u(x)|^2|e^{-\i (x-y)\cdot A\left(\frac{x+y}{2}\right)}-1|^2}{|x-y|^{3+2s}}dxdy \\
& \leq C \|u\|_{s,A}^2+CJ,
\end{align*}	
where we have set 
$$
J:=\int_{K\times K}\frac{|u(x)|^2|e^{-\i (x-y)\cdot A\left(\frac{x+y}{2}\right)}-1|^2}{|x-y|^{3+2s}}dxdy.
$$ 
We now prove that $J\leq C\|u\|_{L^2}^2$, which ends the proof.
We have
\begin{align*}
J &=\int_{K}|u(x)|^2\int_{K\cap \{|x-y|\geq 1\}}\frac{|e^{-\i (x-y)\cdot A\left(\frac{x+y}{2}\right)}-1|^2}{|x-y|^{3+2s}}dxdy  \\
&\quad+\int_{K}|u(x)|^2\int_{K\cap \{|x-y|\leq 1\}}\frac{|e^{-\i (x-y)\cdot A\left(\frac{x+y}{2}\right)}-1|^2}{|x-y|^{3+2s}}dxdy \\
&\leq C\int_{K}|u(x)|^2\int_{K\cap \{|x-y|\geq 1\}}\frac{1}{|x-y|^{3+2s}}dxdy  \\
&\quad+C\int_{K}|u(x)|^2\int_{K\cap \{|x-y|\leq 1\}}\frac{1}{|x-y|^{1+2s}}dxdy,
\end{align*}
where in the last line we used that 
$$
|e^{-\i (x-y)\cdot A\left(\frac{x+y}{2}\right)}-1|^2
\leq C|x-y|^2,\quad \text{for $|x-y|\leq 1$,\,\, $x,y\in K$,}
$$ 
since $A$ is locally bounded. The proof is now complete.
\end{proof}
\begin{lemma}
	\label{gen-bound}
Let $\{A_n\}_{n\in\N}$ be a sequence of uniformly locally bounded functions 
$A_n:\R^3\to\R^3$ with locally bounded gradient and, for any $n\in\N$, $u_n\in H^s_{A_n}(\R^3,\C)$ be such that 
$$
\sup_{n\in\N} \|u_n\|_{s,A_n}<\infty.
$$
Then, up to a subsequence, $\{u_n\}_{n\in\N}$ converges strongly to some function $u$ in $L^q(K,\C)$ for
every compact set $K$ and any $q\in[1,6/(3-2s))$.
\end{lemma}
\begin{proof}
Arguing as in the proof of Lemma~\ref{localemb}, the assertion follows by 
\cite[Corollary 7.2]{DPV}.
\end{proof}

\begin{lemma}[Magnetic Sobolev embeddings]
	\label{embedd}
	The injection
	$$
	H^s_A(\R^3,\C)\hookrightarrow L^p(\R^3,\C) 
	$$
	is continuous for every $2\leq p\leq \frac{6}{3-2s}$. Furthermore, the injection
	$$
	H^s_A(\R^3,\C)\hookrightarrow L^p(K,\C)
	$$ 
	is compact for every $1\leq p<\frac{6}{3-2s}$ and any compact set $K\subset \R^3$.
\end{lemma}
\begin{proof}
By combining Remark~\ref{diamremark} with the continuous injection $H^s(\R^3)\hookrightarrow L^{6/(3-2s)}(\R^3)$ 
(see \cite[Theorem 6.5]{DPV}) yields
\begin{equation}
\label{sobol}
\|u\|_{L^{6/(3-2s)}(\R^3)}\leq C \Big(\int_{\R^6}\frac{|e^{-\i (x-y)\cdot A\left(\frac{x+y}{2}\right)}u(x)-u(y)|^2}{|x-y|^{3+2s}}dxdy\Big)^{1/2}
\quad\text{for all $u\in H^s_A(\R^3,\C)$}.
\end{equation}
Whence, by  interpolation the first assertion immediately follows.\ For the compact embedding, 
taking into account Lemma~\ref{localemb}, the assertion follows by \cite[Corollary 7.2]{DPV}.
\end{proof}

\begin{lemma}[Vanishing]
	\label{lionslemma}
	Let $\{u_{n}\}_{n\in\N}$ be a bounded sequence in $H^{s}(\mathbb R^{3})$ and 
	assume that, for some $R>0$ and $2\leq q <  \frac{6}{3-2s}$, there holds 
	$$
	\lim_n\sup_{\xi\in \R^3}\int_{B(\xi,R)}|u_{n}|^{q}dx=0.
	$$
	Then $u_{n}\to 0$ in $L^{p}(\mathbb R^{3})$ for $2< p <\frac{6}{3-2s}$.
\end{lemma}
\begin{proof}
See \cite[Lemma 2.3]{DSS}.
\end{proof}

\begin{lemma}[Localized Sobolev inequality]
	\label{localized}
Let $\xi\in\R^3$ and $R>0$. Then, for $u\in H^s(B_R(\xi))$, 
$$
\|u\|_{L^\frac{6}{3-2s}(B_R(\xi))}\leq C(s)\left(\frac{1}{R^{2s}}\int_{B_R(\xi)} |u(x)|^2dx+
\int_{B_R(\xi)\times B_R(\xi)}\frac{|u(x)-u(y)|^2}{|x-y|^{3+2s}}dydx\right)^{1/2}
$$
for some constant $C(s)>0$. In particular for every $1\leq p\leq \frac{6}{3-2s}$ there holds
$$
\|u\|_{L^p(B_R(\xi))}\leq C(s,R) \|u\|_{H^s(B_R(\xi))}
$$
for some constant $C(s,R)>0$ and all $u\in H^s(B_R(\xi))$.
	\end{lemma}
	\begin{proof}
See \cite[Proposition 2.5]{BP} for the first inequality. The second inequality immediately follows.
	\end{proof}
	
\begin{lemma}[Cut-off estimates]
	\label{cut}
	Let $u\in H^s_A(\R^3,\C)$ and $\varphi\in C^{0,1}(\R^3)$ with $0\leq \varphi\leq 1$. Then,
	for every pair of measurable sets $E_1,E_2\subset \R^3$, we have
	\begin{align*}
		\int_{E_1\times E_2}\frac{|e^{-\i (x-y)\cdot A\left(\frac{x+y}{2}\right)}\varphi(x)u(x)-\varphi(y)u(y)|^2}{|x-y|^{3+2s}}dxdy
		&\leq C\min\left\{\int_{E_1}|u|^2dx,\int_{E_2}|u|^2dx\right\} \\
		&\quad+C\int_{E_1\times E_2}\frac{|e^{-\i (x-y)\cdot A\left(\frac{x+y}{2}\right)}u(x)-u(y)|^2}{|x-y|^{3+2s}}dxdy,
	\end{align*}
	where $C$ depends on $s$ and on the Lipschitz constant of $\varphi$.
\end{lemma}

\begin{proof}
	The proof follows by arguing as in \cite[Lemma 5.3]{DPV}, where the case $A=0$ and $E_1=E_2$
	is considered.\ For the sake of completeness, we show the details.
	We have
	\begin{align*}
		& \int_{E_1\times E_2}\frac{|e^{-\i (x-y)\cdot A\left(\frac{x+y}{2}\right)}\varphi(x)u(x)-\varphi(y)u(y)|^2}{|x-y|^{3+2s}}dxdy \\
		&\leq  C\int_{E_1\times E_2}\frac{|e^{-\i (x-y)\cdot A\left(\frac{x+y}{2}\right)}u(x)-u(y)|^2}{|x-y|^{3+2s}}dxdy+
		C\int_{E_1\times E_2}\frac{|u(y)|^2|\varphi(x)-\varphi(y)|^2}{|x-y|^{3+2s}}dxdy.
	\end{align*}
	On the other hand, the second integral splits as
	$$
	\int_{E_2}|u(y)|^2\int_{E_1\cap \{|x-y|\leq 1\}}\frac{1}{|x-y|^{1+2s}}dxdy+\int_{E_2}|u(y)|^2\int_{E_1\cap \{|x-y|\geq 1\}}\frac{1}{|x-y|^{3+2s}}dxdy
	\leq C\int_{E_2}|u|^2dy.
	$$
	Analogously, we have
	\begin{align*}
		& \int_{E_1\times E_2}\frac{|\varphi(x)u(x)-e^{\i (x-y)\cdot A\left(\frac{x+y}{2}\right)}\varphi(y)u(y)|^2}{|x-y|^{3+2s}}dxdy \\
		&\leq  C\int_{E_1\times E_2}\frac{|e^{-\i (x-y)\cdot A\left(\frac{x+y}{2}\right)}u(x)-u(y)|^2}{|x-y|^{3+2s}}dxdy+
		C\int_{E_1\times E_2}\frac{|u(x)|^2|\varphi(x)-\varphi(y)|^2}{|x-y|^{3+2s}}dxdy,
	\end{align*}
	and the second term can be estimated as before by $\int_{E_1}|u|^2dx$. The assertion follows.
\end{proof}
	Thus we can prove
	\begin{lemma}[Dicothomy]
		\label{lemma-dic}
		Let $\{u_n\}_{n\in \N}$ be a sequence in $H^s_A(\R^3,\C)$ such that, for some $L>0$,
		$$
		\|u_n\|_{L^p(\R^3)}=1,\qquad \lim_n \|u_n\|_{s,A}^2=L,
		$$
		and let us set
		$$
		\mu_n(x)=|u_n(x)|^2+\int_{\R^3}\frac{|e^{-\i (x-y)\cdot A\left(\frac{x+y}{2}\right)}u_n(x)-u_n(y)|^2}{|x-y|^{3+2s}}dy,\quad x\in\R^3,\,\, n\in\N.
		$$
		Assume that there exists $\beta\in (0,L)$ such 
		that for all $\eps>0$ there exist $\bar R>0$, $\bar n\geq 1$, 
		a sequence of radii $R_n\to+\infty$ and 
		$\{\xi_n\}_{n\in\N}\subset \R^3$ such that for $n\geq \bar n$ 
		\begin{align}
		& \left|\int_{\R^3}\mu_n^1(x)dx-\beta\right|\leq \eps,\qquad 
		\mu_n^1:= 1_{B_{\bar R}(\xi_n)}\mu_n, \nonumber\\
		& \left|\int_{\R^3}\mu_n^2(x)dx-(L-\beta)\right|\leq \eps,\qquad
		\mu_n^2:= 1_{B^c_{R_n}(\xi_n)}\mu_n,  \nonumber\\
		& \int_{\R^3}|\mu_n(x)-\mu_n^1(x)-\mu_n^2(x)|dx\leq \eps.   \label{tre}
		\end{align}
		%
		%
		Then there exist $\{u_n^1\}_{n\in\N},\{u_n^2\}_{n\in\N}\subset 
		H^s_A(\R^3,\C)$ such that ${\rm dist}({\rm supp}(u_n^1),{\rm supp}(u_n^2))\to+\infty$ and 
		\begin{align}
		& \left|\|u_n^1\|_{s,A}^2-\beta\right|\leq \eps, \label{1-est}\\
		& \left|\|u_n^2\|_{s,A}^2-(L-\beta)\right|\leq \eps, \label{2-est}\\
		& \|u_n-u_n^1-u_n^2\|_{s,A}\leq \eps,  \label{3-est} \\
		& \Big|1-\|u_n^1\|_{L^p(\R^3)}^p-\|u_n^2\|_{L^p(\R^3)}^p\Big|\leq \eps \label{normep} 
		\end{align}
		for any $n\ge \bar n$.
	\end{lemma}
	\begin{proof}
		Notice that we have
		\begin{equation}\label{mu1}
		\begin{split}
		\int_{\R^3}\mu_n^1dx &=\int_{B_{\bar R}(\xi_n)}|u_n|^2 dx+\int_{B_{\bar R}(\xi_n)\times B_{{\bar R}}(\xi_n)}
		\frac{|e^{-\i (x-y)\cdot A\left(\frac{x+y}{2}\right)}u_n(x)-u_n(y)|^2}{|x-y|^{3+2s}}dxdy \\
		&\quad+\int_{B_{\bar R}(\xi_n)\times B_{\bar R}^c(\xi_n)}
		\frac{|e^{-\i (x-y)\cdot A\left(\frac{x+y}{2}\right)}u_n(x)-u_n(y)|^2}{|x-y|^{3+2s}}dxdy,
		\end{split}
		\end{equation}
		as well as
		\begin{align*}
		\int_{\R^3}\mu_n^2dx 
		&=\int_{B^c_{R_n}(\xi_n)}|u_n|^2 dx+\int_{B^c_{R_n}(\xi_n)\times B^c_{R_n}(\xi_n)}
		\frac{|e^{-\i (x-y)\cdot A\left(\frac{x+y}{2}\right)}u_n(x)-u_n(y)|^2}{|x-y|^{3+2s}}dxdy \\
		&\quad +\int_{B^c_{R_n}(\xi_n)\times B_{R_n}(\xi_n)}
		\frac{|e^{-\i (x-y)\cdot A\left(\frac{x+y}{2}\right)}u_n(x)-u_n(y)|^2}{|x-y|^{3+2s}}dxdy,
		\end{align*}
		and, from inequality \eqref{tre}, we have,
		for $n\geq \bar n$,
		\begin{align}
		\label{piccola}
		& \int_{\{\bar R\leq |x-\xi_n|\leq R_n\}\times \R^3}\frac{|e^{-\i (x-y)\cdot A\left(\frac{x+y}{2}\right)}u_n(x)-u_n(y)|^2}{|x-y|^{3+2s}}dxdy\leq \eps, \\
		& \int_{\R^3\times \{\bar R\leq |y-\xi_n|\leq R_n\}}\frac{|e^{-\i (x-y)\cdot A\left(\frac{x+y}{2}\right)}u_n(x)-u_n(y)|^2}{|x-y|^{3+2s}}dxdy\leq \eps, 
		\label{piccola2} \\
		\label{piccola3}
		&\int_{\{\bar R\leq |x-\xi_n|\leq R_n\}}|u_n|^2 dx\leq \eps.
		\end{align}
		For every $r>0$, let $\varphi_r\in C^\infty(\R^3)$ be a radially symmetric function such that 
		$\varphi_r=1$ on $B_r(0)$ and $\varphi_r=0$ su $B^c_{2r}(0)$.\
		In light of Lemma~\ref{cut} applied with $E_1=E_2=\R^3$, 
		for any $n\in\N$, we can consider the functions
		$$
		u_n^1:=\varphi_{\bar{R}}(\cdot-\xi_n)u_n\in H^s_A(\R^3,\C),\qquad u_n^2:=(1-\varphi_{R_n/2}(\cdot-\xi_n))u_n\in H^s_A(\R^3,\C).
		$$
		We observe for further usage that the functions $\varphi_{\bar{R}}(\cdot-\xi_n)$ and $1-\varphi_{R_n/2}(\cdot-\xi_n)$
		have a Lipschitz constant which is uniformly bounded with respect to $n$.	
		Moreover, ${\rm dist}({\rm supp}(u_n^1),{\rm supp}(u_n^2))\to\infty$.
		Let us consider $\{u_n^1\}_{n\in\N}$. We have $[u_n^1]_{s,A}^2=\sum_{i=1}^5 I^i_n$, where
		\begin{align*}
		I^1_n& :=\int_{B_{\bar R}(\xi_n)\times B_{\bar R}(\xi_n)}\frac{|e^{-\i (x-y)\cdot A\left(\frac{x+y}{2}\right)}u_n(x)-u_n(y)|^2}{|x-y|^{3+2s}}dxdy, \\
		I^2_n&:=\int_{B_{2\bar R}(\xi_n)\setminus B_{\bar R}(\xi_n)\times B_{2\bar R}(\xi_n)\setminus B_{\bar R}(\xi_n)}\frac{|e^{-\i (x-y)\cdot A\left(\frac{x+y}{2}\right)}u_n^1(x)-u_n^1(y)|^2}{|x-y|^{3+2s}}dxdy, \\
		I^3_n&:=2\int_{B_{2\bar R}(\xi_n)\setminus B_{\bar R}(\xi_n)\times B_{\bar R}(\xi_n)}\frac{|e^{-\i (x-y)\cdot A\left(\frac{x+y}{2}\right)}u_n^1(x)-u_n^1(y)|^2}{|x-y|^{3+2s}}dxdy, \\
		I^4_n&:= 2\int_{B_{2\bar R}(\xi_n)\setminus B_{\bar R}(\xi_n)\times B^c_{2\bar R}(\xi_n)}\frac{|e^{-\i (x-y)\cdot A\left(\frac{x+y}{2}\right)}u_n^1(x)-u_n^1(y)|^2}{|x-y|^{3+2s}}dxdy, \\
		I^5_n& :=2\int_{B_{\bar R}(\xi_n)\times B^c_{2\bar R}(\xi_n)}\frac{|e^{-\i (x-y)\cdot A\left(\frac{x+y}{2}\right)}u_n^1(x)-u_n^1(y)|^2}{|x-y|^{3+2s}}dxdy.
		\end{align*}
		Concerning $I_n^i$ with $i=2,3,4$, since for suitable measurable sets $E_2^i\subset\R^3$ and $c_i>0$,
		\[
		I^i_n= c_i \int_{B_{2\bar R}(\xi_n)\setminus B_{\bar R}(\xi_n)\times E^i_2}\frac{|e^{-\i (x-y)\cdot A\left(\frac{x+y}{2}\right)}u_n^1(x)-u_n^1(y)|^2}{|x-y|^{3+2s}}dxdy,
		\]
		in light of Lemma~\ref{cut} and inequalities \eqref{piccola}-\eqref{piccola3}, we have
		\begin{equation}\label{234}
		\begin{split}
		I_n^i
		& \leq C\left[\int_{B_{2\bar R}(\xi_n)\setminus B_{\bar R}(\xi_n)}|u_n|^2dx+
		\int_{B_{2\bar R}(\xi_n)\setminus B_{\bar R}(\xi_n)\times E^i_2}\frac{|e^{-\i (x-y)\cdot A\left(\frac{x+y}{2}\right)}u_n(x)-u_n(y)|^2}{|x-y|^{3+2s}}dxdy\right]\\
		& \leq C\eps,
		\end{split}		  	
		\end{equation}
		being $B_{2\bar R}(\xi_n)\setminus B_{\bar R}(\xi_n)\subset \{\bar R\leq |x-\xi_n|\leq R_n\}$ for every $n$ large enough.\\
		Concerning $I_n^5$, we have
		\begin{align*}
		I^5_n & =2\int_{B_{\bar R}(\xi_n)\times \{2\bar R\leq |y-\xi_n|\leq R_n\}}\frac{|e^{-\i (x-y)\cdot A\left(\frac{x+y}{2}\right)}u_n^1(x)-u_n^1(y)|^2}{|x-y|^{3+2s}}dxdy \\
		&\quad +2\int_{B_{\bar R}(\xi_n)\times B^c_{R_n}(\xi_n)}\frac{|e^{-\i (x-y)\cdot A\left(\frac{x+y}{2}\right)}u_n^1(x)-u_n^1(y)|^2}{|x-y|^{3+2s}}dxdy.
		\end{align*}
		Then, arguing as in \eqref{234} for $I_n^i$ ($i=2,3,4$)
		we get
		$$
		\int_{B_{\bar R}(\xi_n)\times \{2\bar R\leq |y-\xi_n|\leq R_n\}}
		\frac{|e^{-\i (x-y)\cdot A\left(\frac{x+y}{2}\right)}u_n^1(x)-u_n^1(y)|^2}{|x-y|^{3+2s}}dxdy \leq C\eps,
		$$
		for large $n$. On the other hand, as far as the second term in concerned,  we get 
		\begin{align*}
		& \int_{B_{\bar R}(\xi_n)\times B^c_{R_n}(\xi_n)}\frac{|e^{-\i (x-y)\cdot A\left(\frac{x+y}{2}\right)}u_n^1(x)-u_n^1(y)|^2}{|x-y|^{3+2s}}dxdy 
		=\int_{B_{\bar R}(\xi_n)\times B^c_{R_n}(\xi_n)}\frac{|u_n(x)|^2}{|x-y|^{3+2s}}dxdy,
		\end{align*}	
		since $u^1_n(y)=0$ for all $y\in B^c_{R_n}(\xi_n)$ and $u^1_n(x)=u_n(x)$ for all $x\in B_{\bar R}(\xi_n)$.
		Observe first that if $(x,y)\in B_{\bar R}(\xi_n)\times B^c_{R_n}(\xi_n)$, 
		then $|x-y|\geq R_n-\bar R\to\infty$, as $n\to\infty$. We thus have
		\begin{equation}\label{deltapiccolo}
		\begin{split}
		& \int_{B_{\bar R}(\xi_n)\times B^c_{R_n}(\xi_n)}\frac{|u_n(x)|^2}{|x-y|^{3+2s}}dxdy  \\
		& \leq \frac{1}{(R_n-\bar R)^\delta}\int_{B_{\bar R}(\xi_n)}|u_n(x)|^2\left(\int_{\{|x-y|\geq 1\}}\frac{1}{|x-y|^{3+2s-\delta}}dy\right)dx 
		\leq \frac{C}{(R_n-\bar R)^\delta}\leq C\eps,
		\end{split}
		\end{equation}	
		where $0<\delta<2s$. Here we have used the boundedness of $\{u_n\}_{n\in\N}$ in $L^2(\R^3,\C)$.
		So we have that $[u_n^1]_{s,A}^2= I^1_n +\varsigma_{n,\eps}$ with $\varsigma_{n,\eps}\leq C\eps$
		for $n$ large, which implies on account of \eqref{piccola3} 
		\begin{equation}
		\label{ecco1}
		\|u_n^1\|_{s,A}^2=\int_{B_{\bar R}(\xi_n)}|u_n|^2 dx+ I^1_n +\varsigma_{n,\eps},\quad  \varsigma_{n,\eps}\leq C\eps.
		\end{equation}
		A similar argument involving $\{u_n\}_{n\in\N}$ in place of  $\{u_n^1\}_{n\in\N}$ shows that formula \eqref{mu1} writes as
		\begin{equation}
		\label{ecco2}
		\int_{\R^3}\mu_n^1dx =
		\int_{B_{\bar R}(\xi_n)}|u_n|^2 dx+I^1_n+\hat\varsigma_{n,\eps},\quad  \hat\varsigma_{n,\eps}\leq C\eps.
		\end{equation}
		Indeed, since
		\begin{align*}
		&\int_{B_{\bar R}(\xi_n)\times B_{\bar R}^c(\xi_n)}
		\frac{|e^{-\i (x-y)\cdot A\left(\frac{x+y}{2}\right)}u_n(x)-u_n(y)|^2}{|x-y|^{3+2s}}dxdy\\
		&\quad \leq
		\int_{B_{\bar R}(\xi_n)\times \{\bar R\leq |y-\xi_n|\leq R_n\}}
		\frac{|e^{-\i (x-y)\cdot A\left(\frac{x+y}{2}\right)}u_n(x)-u_n(y)|^2}{|x-y|^{3+2s}}dxdy\\
		&\qquad + C\left[
		\int_{B_{\bar R}(\xi_n)\times B_{R_n}^c(\xi_n)}
		\frac{|u_n(x)|^2}{|x-y|^{3+2s}}dxdy
		+\int_{B_{\bar R}(\xi_n)\times B_{R_n}^c(\xi_n)}
		\frac{|u_n(y)|^2}{|x-y|^{3+2s}}dxdy
		\right],
		\end{align*}
		by \eqref{piccola2} and arguing as in \eqref{deltapiccolo} we can conclude.
		By combining \eqref{ecco1} and \eqref{ecco2} we finally obtain the desired estimate \eqref{1-est}.\\
		Now, concerning $\{u_n^2\}_{n\in\N}$, we have
		$[u_n^2]_{s,A}^2=\sum_{i=1}^5 J^i_n,$
		where we have set
		\begin{align*}
		J^1_n& :=\int_{B^c_{R_n}(\xi_n)\times B^c_{R_n}(\xi_n)}\frac{|e^{-\i (x-y)\cdot A\left(\frac{x+y}{2}\right)}u_n(x)-u_n(y)|^2}{|x-y|^{3+2s}}dxdy, \\
		J^2_n& :=\int_{B_{R_n}(\xi_n)\setminus B_{R_n/2}(\xi_n)\times B_{R_n}(\xi_n)
			\setminus B_{R_n/2}(\xi_n)}\frac{|e^{-\i (x-y)\cdot A\left(\frac{x+y}{2}\right)}u_n^2(x)-u_n^2(y)|^2}{|x-y|^{3+2s}}dxdy, \\
		J^3_n&:=2\int_{B_{R_n}(\xi_n)\setminus B_{R_n/2}(\xi_n)\times B_{R_n/2}(\xi_n)}\frac{|e^{-\i (x-y)\cdot A\left(\frac{x+y}{2}\right)}u_n^2(x)-u_n^2(y)|^2}{|x-y|^{3+2s}}dxdy, \\
		J^4_n&:= 2\int_{B_{R_n}(\xi_n)\setminus B_{R_n/2}(\xi_n)\times B^c_{R_n}(\xi_n)}\frac{|e^{-\i (x-y)\cdot A\left(\frac{x+y}{2}\right)}u_n^2(x)-u_n^2(y)|^2}{|x-y|^{3+2s}}dxdy, \\
		J^5_n& :=2\int_{B_{R_n/2}(\xi_n)\times B^c_{R_n}(\xi_n)}\frac{|e^{-\i (x-y)\cdot A\left(\frac{x+y}{2}\right)}u_n^2(x)-u_n^2(y)|^2}{|x-y|^{3+2s}}dxdy.
		\end{align*}
		Concerning $J_n^i$ with $i=2,3,4$, observe that the integration domains are $B_{R_n}(\xi_n)\setminus B_{R_n/2}(\xi_n)\times E_2^i$, for suitable measurable $E_2^i$'s, and they are subset of $\{\bar R\leq |x-\xi_n|\leq R_n\}\times \R^3$ for $n$ sufficiently large. Thus we can argue as in \eqref{234}.
		Finally,
		$J^5_n$ can be estimated with similar arguments to that used in \eqref{deltapiccolo} and, as for $\mu_n^1$, using also \eqref{piccola3}, we obtain
		\begin{equation*}
		\int_{\R^3}\mu_n^2(x)dx 
		=\int_{B^c_{R_n}(\xi_n)}|u_n|^2 dx+J^1_n+\bar\varsigma_{n,\eps},\quad  \bar\varsigma_{n,\eps}\leq C\eps.
		\end{equation*}
		By combining all these estimates we get \eqref{2-est}
		for any $n$ large.\\
		Conclusion \eqref{3-est} follows
		by \eqref{piccola}-\eqref{piccola3}. In fact, setting 
		$$
		v_n:=u_n-u_n^1-u_n^2=(\varphi_{R_n/2}(\cdot-\xi_n)-\varphi_{\bar{R}}(\cdot-\xi_n))u_n,
		$$ 
		for all $n$, inequality \eqref{piccola3} yields
		$$
		\int_{\R^3}|v_n|^2dx=\int_{\R^3}(\varphi_{R_n/2}(x-\xi_n)-\varphi_{\bar{R}}(x-\xi_n) )^2|u_n|^2 dx
		\leq\int_{\{\bar R\leq |x-\xi_n|\leq R_n\}}|u_n|^2 dx\leq \eps.
		$$
		Furthermore, $[v_n]_{s,A}^2=\sum_{i=1}^4 K_n^i$, where
		\begin{align*}
		K^1_n&:=\int_{B_{R_n}(\xi_n)\setminus B_{\bar R}(\xi_n)\times B_{R_n}(\xi_n)\setminus B_{\bar R}(\xi_n)}\frac{|e^{-\i (x-y)\cdot A\left(\frac{x+y}{2}\right)}v_n(x)-v_n(y)|^2}{|x-y|^{3+2s}}dxdy, \\
		K^2_n&:=2\int_{B_{R_n}(\xi_n)\setminus B_{\bar R}(\xi_n)\times B_{\bar R}(\xi_n)}\frac{|e^{-\i (x-y)\cdot A\left(\frac{x+y}{2}\right)}v_n(x)-v_n(y)|^2}{|x-y|^{3+2s}}dxdy, \\
		K^3_n&:= 2\int_{B_{R_n}(\xi_n)\setminus B_{\bar R}(\xi_n)\times B^c_{R_n}(\xi_n)}\frac{|e^{-\i (x-y)\cdot A\left(\frac{x+y}{2}\right)}v_n(x)-v_n(y)|^2}{|x-y|^{3+2s}}dxdy.
		\end{align*}
		Since $v_n=\tilde\varphi u_n$ with 
		$\tilde\varphi:=(\varphi_{R_n/2}(\cdot-\xi_n)-\varphi_{\bar{R}}(\cdot-\xi_n))$,  we can repeat
		the arguments performed in \eqref{234}.
		Concerning the final assertion \eqref{normep}, we have for some $\vartheta>0$,
		\begin{align*}
		1-\|u_n^1\|_{L^p}^p-\|u_n^2\|_{L^p}^p
		&   =
		\int_{\R^3} \left(1-\varphi_{\bar{R}}^p(x-\xi_n)-(1-\varphi_{R_n/2}(x-\xi_n))^p\right)|u_n|^pdx \\
		& \leq \int_{\{\bar R\leq |x-\xi_n|\leq R_n\}}|u_n|^p dx\\
		& \leq \Big(\int_{\{\bar R\leq |x-\xi_n|\leq R_n\}}\!\!\!|u_n|^2 dx\Big)^{\frac{\vartheta p}{2}}
		\Big(\int_{\R^3} |u_n|^{\frac{6}{3-2s}} dx\Big)^{\frac{(1-\vartheta)p(3-2s)}{6}}\\
		&\leq\eps,
		\end{align*}
		in light of \eqref{piccola3} and Lemma~\ref{embedd}.\ This concludes the proof.
	\end{proof}

	\begin{lemma}[Partial Gauge invariance]
		\label{gauge2}
		Let $\xi\in\mathbb{R}^3$ and $u\in H^s_A(\R^3,\C)$. For $\eta\in\R^3$, let us set
		$$
		v(x)=e^{\i \eta\cdot x} u(x+\xi),\quad x\in\R^3.
		$$
		Then $v\in H^s_{A_\eta}(\R^3,\C)$ and 
		$$
		\|u\|_{s,A}= \|v\|_{s,A_\eta},\quad\text{where $A_\eta:=A(\cdot+\xi)+\eta$}.
		$$
	\end{lemma}
	\begin{proof}
		Of course $\|v\|_{L^2}=\|u\|_{L^2}$.\ Moreover, a change of variables yields
		\begin{align*}
		\int_{\R^6}
		\frac{|e^{-\i (x-y)\cdot A_\eta\left(\frac{x+y}{2}\right)}v(x)-v(y)|^2}{|x-y|^{3+2s}}dxdy
		&=
		\int_{\R^6}
		\frac{|e^{-\i (x-y)\cdot A\left(\frac{x+y}{2}+\xi\right)} u(x+\xi)- u(y+\xi)|^2}{|x-y|^{3+2s}}dxdy\\
		&=
		\int_{\R^6}
		\frac{|e^{-\i (x-y)\cdot A\left(\frac{x+y}{2}\right)} u(x)- u(y)|^2}{|x-y|^{3+2s}}dxdy,
		\end{align*}
		which yields the assertion.
	\end{proof}

\noindent
If $A$ is linear, then, taking $\eta=-A(\xi)$ in Lemma~\ref{gauge2},  we get $A_\eta=A$ and hence

\begin{lemma}[Partial Gauge invariance]
		\label{gauge}
		Let $\xi\in\mathbb{R}^3$ and $u\in H^s_A(\R^3,\C)$. Assume that $A$ is linear and let us set 
		$$
		v(x)=e^{-\i A(\xi)\cdot x} u(x+\xi),\quad x\in\R^3.
		$$
		Then $v\in H^s_A(\R^3,\C)$ and $\|u\|_{s,A}= \|v\|_{s,A}$.
	\end{lemma}

\section{Existence of minimizers}\label{sect4}\label{sec4}
\noindent
Let $2<p<6/(3-2s)$ and  consider the minimization problem \eqref{MA}. First of all observe that by Sobolev embedding, ${\mathscr M}_A>0$.
Once a solution to \eqref{MA} exists, 
due to the Lagrange Multiplier Theorem, there is $\lambda\in\R$ such that 
\begin{align*}
& \frac{c_s}{2}{\Re}\int_{\R^6}\frac{\big(e^{-\i (x-y)\cdot A\left(\frac{x+y}{2}\right)}u(x)-u(y)\big)\overline{\big(e^{-\i (x-y)\cdot A\left(\frac{x+y}{2}\right)}v(x)-v(y)\big)}}{|x-y|^{3+2s}}dxdy  \\
&+\Re\int_{\R^3} u\bar v dx=\lambda \Re\int_{\R^3} |u|^{p-2}u\bar v dx,\quad \text{for all $v\in H^s_A(\R^3,\C)$}.
\end{align*}
A multiple of $u$ removes the Lagrange multiplier $\lambda$ and provides a 
weak solution to \eqref{prob}.
Moreover, if we set
$$
{\mathscr M}_A(\lambda):=\inf_{u\in \Ss(\lambda)} \|u\|_{s,A}^2,
$$
where
$$
\Ss(\lambda):=\left\{u\in H^s_A(\R^3,\C):\int_{\R^3}|u|^p dx=\lambda\right\},
$$
we have that for every $\lambda>0$
\begin{equation}
\label{scala}
{\mathscr M}_A(\lambda)=\lambda^{\frac{2}{p}}{\mathscr M}_A.
\end{equation}


\subsection{Subcritical symmetric case}
\noindent
Let $2<p<\frac{6}{3-2s}$ and consider the problem
$$
{\mathscr M}_{A,r}=\inf_{u\in \Ss_r} \|u\|_{s,A}^2,
$$
where
$$ 
\Ss_r=\left\{u\in H^s_{A,{\rm rad}}(\R^3,\C): \int_{\R^3}|u|^p dx=1\right\}.
$$
First we give the following preliminary result.
\begin{lemma}[Compact radial embedding]
	\label{comp-emb}
	For every $2<q<6/(3-2s),$ the mapping
	$$
	H^s_{A,{\rm rad}}(\R^3,\C)\ni u\mapsto |u|\in L^q(\R^3),
	$$
	is compact.
	\end{lemma}
	\begin{proof}
		By Lemma~\ref{diamagnetic}, namely the
		Diamagnetic inequality, we know that the mapping
		$$
		H^s_A(\R^3,\C)\ni u\mapsto |u|\in H^s(\R^3,\R),
		$$ 
		is continuous. Then, the assertion follows directly by \cite[Theorem II.1]{Lions-c}.
		\end{proof}
\noindent
We are ready to prove (\ref{i12}) of Theorem \ref{main}
\begin{theorem}[Existence of radial minimizers]
	\label{Ex1-rad}
	For any $2<p<6/(3-2s)$, the minimization problem ${\mathscr M}_{A,r}$
	admits a solution.\ In particular, there exists a nontrivial 
	radially symmetric weak solution $u\in H^s_{A,{\rm rad}}(\R^3,\C)$ to the problem \eqref{prob}.
\end{theorem}
\begin{proof}
	Let $\{u_n\}_{n\in\N}\subset \Ss_r$ be a minimizing sequence for ${\mathscr M}_{A,r}$, namely $\|u_n\|_{L^p(\R^3)}=1$ for all $n$ and
	$\|u_n\|_{s,A}^2\to {\mathscr M}_{A,r}$, as $n\to\infty$. Then, up to a subsequence,  it converges weakly to some radial function $u$.\ On account
	of Lemma~\ref{embedd}, $u_n\to u$ a.e.\ up to a subsequence.
	By Lemma~\ref{comp-emb}, up to a subsequence
	$\{|u_n|\}_{n\in\N}$ converges strongly to some $v$ in $L^{q}(\R^3)$
	for every $2<q<6/(3-2s)$.\	Of course, $v=|u|$ by pointwise convergence. 
	In particular we can pass to the limit into the constraint $\|u_n\|_{L^p(\R^3)}=1$ to get $\|u\|_{L^p(\R^3)}=1$. Then $u$ is a solution to 
	${\mathscr M}_{A,r}$, since by virtue of Fatou Lemma 
	\begin{align*}
		{\mathscr M}_{A,r}& \leq 
		\int_{\R^3} |u(x)|^2dx+
		\int_{\R^6}\frac{|e^{-\i (x-y)\cdot A\left(\frac{x+y}{2}\right)}u(x)-u(y)|^2}{|x-y|^{3+2s}}dxdy  \\
		& \leq
		\liminf_n	\left(\int_{\R^3} |u_n(x)|^2dx+
		\int_{\R^6}\frac{|e^{-\i (x-y)\cdot A\left(\frac{x+y}{2}\right)}u_n(x)-u_n(y)|^2}{|x-y|^{3+2s}}dxdy\right)
		={\mathscr M}_{A,r}.
	\end{align*}
	This concludes the proof.
\end{proof}

\subsection{Subcritical case}

\noindent
In this subsection we study the minimization problem \eqref{MA} in the case $2<p<\frac{6}{3-2s}$.

\subsubsection{Constant magnetic field case}
Let us consider \eqref{MA} under the assumption that $A:\R^3\to\R^3$ is linear. The local case was extensively 
studied in \cite{EL} for the magnetic potential
$$
A(x_1,x_2,x_3)=\frac{b}{2} (-x_2,x_1,0),\quad b\in\R\setminus\{0\}.
$$

\noindent
Hence we can prove (\ref{iii12}) of Theorem \ref{main}.
\begin{theorem}[Existence of minimizers, I]
	\label{Ex1}
Assume that the potential $A:\R^3\to\R^3$ is linear. Then, for any $2<p<\frac{6}{3-2s}$ the minimization problem \eqref{MA}
admits a solution. 
\end{theorem}
\begin{proof}
Let $\{u_n\}_{n\in\N}\subset \Ss$ be a minimizing sequence for ${\mathscr M}_A$, namely $\|u_n\|_{L^p}=1$ for all $n$ and
$\|u_n\|_{s,A}^2\to {\mathscr M}_A$, as $n\to\infty$.\ We want to develop a concentration compactness argument \cite{L1}
on the measure of density defined by
$$
\mu_n(x):=|u_n(x)|^2+\int_{\R^3}\frac{|e^{-\i (x-y)\cdot A\left(\frac{x+y}{2}\right)}u_n(x)-u_n(y)|^2}{|x-y|^{3+2s}}dy,\quad x\in\R^3,\,\, n\in\N.
$$
Notice that $\{\mu_n\}_{n\in\N}\subset L^1(\R^3)$ and, since $\|u_n\|_{s,A}^2={\mathscr M}_A+o_n(1)$,
$$
\sup_{n\in\N}\int_{\R^3} \mu_n(x)dx<\infty.
$$
More precisely, we shall apply \cite[Lemma I.1]{L1} by taking $\rho_n=\mu_n$. 
Only vanishing, dichotomy or tightness (yielding compactness) are possible.\ Vanishing can be ruled out. 
In fact, assume by contradiction that, for all $R>0$ fixed, there holds
$$
\lim_{n}\sup_{\xi\in\R^N} \int_{B_R(\xi)} \mu_n(x)dx=0,
$$
namely
$$
\lim_{n}\sup_{\xi\in\R^N} \left(\int_{B_R(\xi)} |u_n(x)|^2dx+
\int_{B_R(\xi)\times\R^3}\frac{|e^{-\i (x-y)\cdot A\left(\frac{x+y}{2}\right)}u_n(x)-u_n(y)|^2}{|x-y|^{3+2s}}dxdy\right)=0.
$$
By Remark~\ref{diamremark} it follows that 
$$
\lim_{n}\sup_{\xi\in\R^N} \left(\int_{B_R(\xi)} |u_n(x)|^2dx+
\int_{B_R(\xi)\times\R^3}\frac{||u_n(x)|-|u_n(y)||^2}{|x-y|^{3+2s}}dxdy\right)=0.
$$
In particular, we get
$$
\lim_{n}\sup_{\xi\in\R^N} \||u_n|\|_{H^s(B_R(\xi))}^2
=0
$$
and this implies, by virtue of Lemma~\ref{localized}, that for any $R>0$
$$
\lim_{n}\sup_{\xi\in\R^N} \int_{B_R(\xi)} |u_n(x)|^pdx=0.
$$
Thus, in light of Lemma~\ref{lionslemma}, $u_n\to 0$ in $L^p$ which violates
the constraint $\|u_n\|_{L^p}=1$. Whence, vanishing cannot occur.\\
We now exclude the dicothomy. 
According to \cite[Lemma I.1]{L1}, this, precisely, means
that there exists $\beta\in (0,{\mathscr M}_A)$ such 
 that for all $\eps>0$ there are $\bar R>0$, $\bar n\geq 1$, 
 a sequence of radii $R_n\to+\infty$ and 
 $\{\xi_n\}_{n\in\N}\subset \R^3$ such that for $n\geq \bar n$ 
 \begin{align*}
 	& \left|\int_{\R^3}\mu_n^1(x)dx-\beta\right|\leq \eps,\qquad 
 	\mu_n^1(x):= 1_{B_{\bar R}(\xi_n)}\mu_n, \\
 	& \left|\int_{\R^3}\mu_n^2(x)dx-({\mathscr M}_A-\beta)\right|\leq \eps,\qquad
 	\mu_n^2(x):= 1_{B^c_{R_n}(\xi_n)}\mu_n,  \\
 	& \int_{\R^3}|\mu_n(x)-\mu_n^1(x)-\mu_n^2(x)|dx\leq \eps.   
 \end{align*}
Then, by virtue of Lemma~\ref{lemma-dic},
there exist two sequences $\{u_n^1\}_{n\in\N},\{u_n^2\}_{n\in\N}\subset 
	H^s_A(\R^3,\C)$ such that  ${\rm dist}({\rm supp}(u_n^1),{\rm supp}(u_n^2))\to+\infty$ and 
	\begin{align}
	& \left|\|u_n^1\|_{s,A}^2-\beta\right|\leq \eps,    \label{vicine-a} \\
	& \left|\|u_n^2\|_{s,A}^2-({\mathscr M}_A-\beta)\right|\leq \eps,  \label{vicine-b}\\
	& \left|1-\|u_n^1\|_{L^p}^p-\|u_n^2\|_{L^p}^p\right|\leq \eps, \label{vicine-p}
	\end{align}
	for any $n\ge \bar n$. Up to a subsequence, in view of \eqref{vicine-p}, 
	there exist $\vartheta_\eps,\omega_\eps\in (0,1)$ such that
	$$
	\|u_n^1\|_{L^p}^p=:\vartheta_{n,\eps}\to\vartheta_\eps,
	\qquad 
	\|u_n^2\|_{L^p}^p=:\omega_{n,\eps}\to \omega_\eps,
	\qquad
	|1-\vartheta_{\eps}-\omega_\eps|\leq \eps,\quad\text{as $n\to\infty$.}
	$$ 
	Notice that $\vartheta_\eps$ does not converge to $1$ as $\eps\to 0$, otherwise by \eqref{scala} and \eqref{vicine-a},
	for $\eps$ small we get
	\begin{equation*}
	\beta+\eps
	\geq \limsup_{n}\|u_n^1\|_{s,A}^2
	\geq \limsup_n {\mathscr M}_A(\vartheta_{n,\eps})
	={\mathscr M}_A \vartheta_\eps^{2/p}>\beta+\eps.
	\end{equation*}
	Of course $\vartheta_\eps$ does not converge to $0$ either, as $\eps\to 0$, otherwise $\omega_\eps\to 1$
and a contradiction would again follow by arguing as above on $u_n^2$ and  using \eqref{vicine-b}.\
	Whence, by means of \eqref{scala}, \eqref{vicine-a}, \eqref{vicine-b}, and since $\lambda^{2/p}+(1-\lambda)^{2/p}>1$ for any $\lambda\in(0,1)$,
	if $\eps$ is small enough
	\begin{align*}
	{\mathscr M}_A+2\eps
	&\geq \limsup_{n}\left(\|u_n^1\|_{s,A}^2+\|u_n^2\|_{s,A}^2\right)
	\geq \limsup_n \left({\mathscr M}_A(\vartheta_{n,\eps})+{\mathscr M}_A(\omega_{n,\eps})\right)\\
	&={\mathscr M}_A \left(\vartheta_\eps^{2/p}+\omega_\eps^{2/p}\right)>{\mathscr M}_A+2\eps,
	\end{align*}
a contradiction.\ This means that 
tightness needs to occur, namely there exists
a sequence $\{\xi_n\}_{n\in\N}$ such that for all $\eps>0$ there exists $R>0$ with
$$
\int_{B^c_R(\xi_n)} |u_n(x)|^2dx+
\int_{B^c_R(\xi_n)\times\R^3}
\frac{|e^{-\i (x-y)\cdot A\left(\frac{x+y}{2}\right)}u_n(x)-u_n(y)|^2}{|x-y|^{3+2s}}dxdy<\eps
$$
for any $n$. In particular, setting $\bar u_n(x):=u_n(x+\xi_n)$, for all $\eps>0$ there is $R>0$ such that
\begin{equation}
\label{ext-d}
\sup_{n\in\N}\int_{B^c_R(0)} |\bar u_n(x)|^2dx<\eps.
\end{equation}
Let us consider
$$
v_n(x):=e^{-\i A(\xi_n)\cdot x} \bar u_n(x),\quad x\in\R^3.
$$
Since, by Lemma~\ref{gauge}, $\|v_n\|_{s,A}=\|u_n\|_{s,A}$, we have that $\{v_n\}_{n\in\N}$ is bounded in $H^s_A(\R^3,\C)$.
Notice also that, since $|v_n(x)|=|\bar u_n(x)|$ for a.e.\ $x\in\R^3$ and any $n\in\N$, by \eqref{ext-d} we have
that for all $\eps>0$ there is $R>0$ such that
\begin{equation}\label{vn}
\sup_{n\in\N}\int_{B^c_R(0)} |v_n(x)|^2dx<\eps.
\end{equation}
Thus, in view of the compact injection 
provided by Lemma~\ref{embedd}, up to a subsequence, $\{v_n\}_{n\in\N}$ converges weakly,
strongly in $L^2(B_R(0),\C)$ and point-wisely to some function $v$.
Moreover, by \eqref{vn},
it follows that 
$v_n\to v$ strongly in $L^2(\R^3,\C)$ as well as
in $L^q(\R^3,\C)$ for any $2<q<6/(3-2s)$, via interpolation. Hence $\|v\|_{L^p}=1$. 
Hence, by Fatou's lemma, we have
$$
{\mathscr M}_A\leq \|v\|^2_{s,A}\leq \liminf_n \|v_n\|_{s,A}^2=\liminf_n \|u_n\|_{s,A}^2={\mathscr M}_A,
$$
which proves the existence of a minimizer.
\end{proof}

\subsubsection{Variable magnetic field case}

We now prove (\ref{ii12}) of Theorem \ref{main}.

\begin{theorem}[Existence of minimizers, II]
	\label{Ex2}
	Assume that the potential $A:\R^3\to\R^3$ satisfies assumption $\mathscr{A}$ and that
	\begin{equation}
	\label{stretta}
	{\mathscr M}_A<\inf_{\Xi\in \mathscr{X}} {\mathscr M}_{A_\Xi}.
	\end{equation}
	Then, for any $2<p<\frac{6}{3-2s},$ the minimization problem \eqref{MA}
	admits a solution.
\end{theorem}
\begin{proof}
		By arguing as in the proof of Theorem~\ref{Ex1}, if $\{u_n\}_{n\in\N}$ is a minimizing sequence for ${\mathscr M}_A$,  we can
		find a sequence $\{\xi_n\}_{n\in\N}$
		such that for all $\eps>0$ there exists $R>0$ with
		$$
		\int_{B^c_R(\xi_n)} |u_n(x)|^2dx+
		\int_{B^c_R(\xi_n)\times\R^3}
		\frac{|e^{-\i (x-y)\cdot A\left(\frac{x+y}{2}\right)}u_n(x)-u_n(y)|^2}{|x-y|^{3+2s}}dxdy<\eps
		$$
		for any $n$.
		In particular, setting again $\bar u_n(x):=u_n(x+\xi_n)$, for all $\eps>0$ there is $R>0$ such that
		\begin{equation*}
		\sup_{n\in\N}\int_{B^c_R(0)} |\bar u_n(x)|^2dx<\eps.
		\end{equation*}
		Assume by contradiction that the sequence $\{\xi_n\}_{n\in\N}$ is unbounded. Then, since $A$ satisfies condition
		${\mathscr A}$, there exists a sequence $\{H_n\}_{n\in\N}\subset\R^3$ such that \eqref{limA} holds. We thus consider
		the sequence
		$$
		v_n(x):=e^{\i H_n\cdot x} \bar u_n(x), \quad x\in\R^3.
		$$
		By virtue of Lemma~\ref{gauge2} it follows that
		$$
		\sup_{n\in\N} \|v_n\|_{s,A_n}=\sup_{n\in\N}\|u_n\|_{s,A}<\infty,\quad A_n(x)=A(x+\xi_n)+H_n.
		$$
		Then, by combining Lemma~\ref{gen-bound} with
		\begin{equation*}
		\sup_{n\in\N}\int_{B^c_R(0)} |v_n(x)|^2dx<\eps,
		\end{equation*}
		up to a subsequence, $\{v_n\}_{n\in\N}$  is strongly convergent in $L^q(\R^3)$ for all $q\in[2,6/(3-2s))$ to some function $v$
		which satisfies the constraint $\|v\|_{L^p}=1$. By combining Lemma~\ref{gauge2} with Fatou's Lemma
		and \eqref{stretta}, we get	
		\begin{align*}
		{\mathscr M}_{A_\Xi} \leq \|v\|^2_{s,A_{\Xi}} 
		&=	\int_{\R^3} |v|^2 dx
		+\int_{\R^6}
		\frac{|e^{-\i (x-y)\cdot A_{\Xi}\left(\frac{x+y}{2}\right)}v(x)-v(y)|^2}{|x-y|^{3+2s}}dxdy  \\
		&\leq \lim_n \int_{\R^3} |v_n|^2 dx
		+\liminf_n \int_{\R^6}
		\frac{|e^{-\i (x-y)\cdot A_n\left(\frac{x+y}{2}\right)}v_n(x)-v_n(y)|^2}{|x-y|^{3+2s}}dxdy \\
		&= \lim_n \int_{\R^3} |u_n|^2 dx
		+\liminf_n \int_{\R^6}
		\frac{|e^{-\i (x-y)\cdot A\left(\frac{x+y}{2}\right)}u_n(x)-u_n(y)|^2}{|x-y|^{3+2s}}dxdy \\
		&={\mathscr M}_{A}<\inf_{\Xi\in \mathscr{X}} {\mathscr M}_{A_\Xi}\leq {\mathscr M}_{A_\Xi},
		\end{align*}
		a contradiction.\ Therefore, it follows that  $\{\xi_n\}_{n\in\N}$ is bounded. The assertion then 
		immediately follows arguing on the original sequence $\{u_n\}_{n\in\N}$.
	\end{proof}
%


\subsection{Critical case}\label{subscrit}
\noindent
Let $D^s_A(\R^3,\C)$ be the completion of $C^\infty_c(\R^3,\C)$ with respect to the semi-norm $[\cdot]_{s,A}$.\
The functions of $D^s_A(\R^3,\C)$ satisfy the Sobolev inequality stated in formula~\eqref{sobol}.\
The space $D^s_A(\R^3,\C)$ is a real Hilbert space with respect to the scalar product 
$$
(u,v)_{s,A}:=\frac{c_s}{2}{\Re}\int_{\R^6}\frac{\left(e^{-\i (x-y)\cdot A\left(\frac{x+y}{2}\right)}u(x)-u(y)\right)\overline{\left(e^{-\i (x-y)\cdot A\left(\frac{x+y}{2}\right)}v(x)-v(y)\right)}}{|x-y|^{3+2s}}dxdy.
$$
\noindent
We consider the minimization problem \eqref{MAC}. 
Of course, by density, we have
$$
{\mathscr M}_A^c=\inf_{u\in \Ss^c\cap C^\infty_c(\R^3,\C)} [u]_{s,A}^2,
\quad \,\,\,
{\mathscr M}_0^c=\inf_{u\in \Ss^c_0\cap C^\infty_c(\R^3,\C)} [u]_{s,0}^2.
$$
where $\Ss^c_0=\{u\in D^s(\R^3,\C):\|u\|_{L^{6/(3-2s)}}=1\}$.
Moreover, since $[|u|]_{s,0}\leq [u]_{s,0}$, we have
\begin{equation}
\label{real-rap}
{\mathscr M}_0^c=\inf_{u\in \Ss^c_0\cap C^\infty_c(\R^3,\R)} [u]_{s,0}^2.
\end{equation}
\begin{remark}\rm
	\label{classif}
It is known \cite{classif,costav} that all the real valued fixed sign solutions to ${\mathscr M}_0^c$ are given by
\begin{equation*}
\label{Talentiane}
{\mathscr U}_{z,\eps} (x)=d_{s}\left(\frac{\eps}{\eps^2+|x-z|^2}\right)^{\frac{3-2s}{2}}
\end{equation*}
for arbitrary $\eps>0$, $z\in \R^3$ and that these are also the unique fixed sign solutions to
\begin{equation*}
(-\Delta)^s u=u^{\frac{3+2s}{3-2s}}\quad \text{in $\R^3$.}
\end{equation*}
\end{remark}
\noindent
We now prove the following crucial lemma

\begin{lemma}
	\label{livelli}
It holds ${\mathscr M}_A^c={\mathscr M}_0^c$.
\end{lemma}
\begin{proof}
Let $\eps>0$ and $u\in C^\infty_c(\R^3,\R)$ be such that
$$
\int_{\R^3}|u|^{6/(3-2s)} dx=1,\quad \,\,
[u]_{s,0}^2\leq {\mathscr M}_0^c+\eps,
$$	
in light of formula~\eqref{real-rap} for ${\mathscr M}_0^c$.\ Consider now the scaling
$$
u_\sigma(x)=\sigma^{-\frac{3-2s}{2}}u\Big(\frac{x}{\sigma}\Big),\quad \sigma>0,\,\,\, x\in\R^3.
$$
It is readily checked that 
$$
\int_{\R^3}|u_\sigma|^{6/(3-2s)} dx=\int_{\R^3}|u|^{6/(3-2s)} dx=1,\quad 
[u_\sigma]_{s,0}=[u]_{s,0},\quad \text{for all $\sigma>0$.}
$$
There holds 
that
\[
[u_\sigma]^2_{s,A} = \int_{\R^6}\frac{|e^{-\i \sigma (x-y)\cdot A\left(\sigma\frac{x+y}{2}\right)}u(x)-u(y)|^2}{|x-y|^{3+2s}}dxdy.
\]
Then, we compute 
\begin{align*}
[u_\sigma]^2_{s,A}-[u]_{s,0}^2 &=\int_{\R^6}\frac{|e^{-\i \sigma (x-y)\cdot A\left(\sigma\frac{x+y}{2}\right)}u(x)-u(y)|^2-|u(x)-u(y)|^2}{|x-y|^{3+2s}}dxdy \\
& =\int_{\R^6} \varTheta_\sigma(x,y)dxdy=\int_{K\times K} \varTheta_\sigma(x,y)dxdy,
\end{align*}
where $K$ is the compact support of  $u$ and 
\begin{align*}
\varTheta_\sigma(x,y)
&:=\frac{2\Re\left(\Big(1-e^{-\i \sigma (x-y)\cdot A\left(\sigma\frac{x+y}{2}
		\right)}\Big)u(x) u(y)\right)}{|x-y|^{3+2s}}\\
&=\frac{2\big(1-\cos\big(\sigma (x-y)\cdot A\big(\sigma\frac{x+y}{2}\big)\big)\big)u(x)u (y)}{|x-y|^{3+2s}},
\end{align*}
a.e. in $\R^6$.
Of course $\varTheta_\sigma(x,y)\to 0$ for a.e.\ $(x,y)\in\R^6$ as $\sigma\to 0$.
Since $A$ is locally 
bounded then
\[
1-\cos\Big(\sigma (x-y)\cdot A\Big(\sigma\frac{x+y}{2}\Big)\Big) \leq C |x-y|^2
\quad x,y\in K.
\]
Therefore, since $u$ is bounded, it follows that for some $C>0$
\begin{align*}
& |\varTheta_\sigma(x,y)|\leq \frac{C}{|x-y|^{1+2s}},\quad\text{for $x,y\in K$ with $|x-y|<1$,} \\
& |\varTheta_\sigma(x,y)|\leq \frac{C}{|x-y|^{3+2s}},\quad\text{for $x,y\in K$ with $|x-y|\geq 1$.}
\end{align*}
Then overall, we have  
$$
|\varTheta_\sigma(x,y)|\leq w(x,y),\quad w(x,y)=C\min\left\{\frac{1}{|x-y|^{1+2s}},\frac{1}{|x-y|^{3+2s}}\right\}\quad\text{for $x,y\in K$,}
$$ 
for a suitable constant $C>0$.
Notice that $w\in L^1(K\times K)$, since
\begin{align*}
\int_{K\times K} w(x,y) dxdy&=\int_{(K\times K)\cap\{|x-y|<1\}} w(x,y) dxdy+\int_{(K\times K)\cap\{|x-y|\geq 1\}}  w(x,y) dxdy \\
& \leq C\int_{\{|z|<1\}} \frac{1}{|z|^{1+2s}} dz+C\int_{\{|z|\geq 1\}}  \frac{1}{|z|^{3+2s}}dz<\infty.
\end{align*}
Then, by the Dominated Convergence Theorem, we obtain 
$$
{\mathscr M}_A^c\leq \lim_{\sigma\to 0}[u_\sigma]^2_{s,A}=[u]_{s,0}^2 \leq {\mathscr M}_0^c+\eps,
$$
hence ${\mathscr M}_A^c\leq {\mathscr M}_0^c$ by the arbitrariness of $\eps$. Since the opposite 
inequality is trivial through the Diamagnetic inequality, the desired assertion follows.
\end{proof}
\noindent
Thus we can prove (\ref{i13}) of Theorem \ref{main-2}.

\begin{theorem}[Representation of solutions]
	Assume that ${\mathscr M}_A^c$ admits a solutions $u\in D^s_A(\R^3,\C)$. Then there exist
	$z\in\R^3$, $\eps>0$ and a function $\vartheta_A:\R^3\to\R$ such that
	$$
	u(x)=d_{s}\left(\frac{\eps}{\eps^2+|x-z|^2}\right)^{\frac{3-2s}{2}}e^{\i \vartheta_A(x)},\quad x\in\R^3.
	$$ 
\end{theorem}
\begin{proof}
	If  $u\in D^s_A(\R^3,\C)$ is a solution to ${\mathscr M}_A^c$, then by the Diamagnetic inequality and Lemma~\ref{livelli},
	$$
	{\mathscr M}_A^c={\mathscr M}_0^c\leq [|u|]_{s,0}^2\leq [u]_{s,A}^2={\mathscr M}_A^c.
	$$
	Then, it follows that ${\mathscr M}_0^c=[|u|]_{s,0}^2$, which implies the assertion by Remark~\ref{classif}.
\end{proof}

\noindent
 For a function $u\in D^s_A(\R^3,\C)$ we define $\Upsilon_u^A:\R^6\to\R$ by setting
$$
\Upsilon_u^A(x,y):=2\Re\Big(|u(x)||u(y)|-e^{-\i (x-y)\cdot A\big(\frac{x+y}{2}\big)}u(x)\bar u(y)\Big),\quad 
\text{a.e.\ in $\R^6$}. 
$$
Finally we have
\begin{theorem}[Nonexistence]
	\label{nonex}
	Assume that for a function $u\in D^s_A(\R^3,\C)$ we have
	\begin{equation}
	\label{Aass-00}
	 \Upsilon^A_u(x,y)>0\text{ on $E\subset\R^6$ with ${\mathcal L}^6(E)>0$.}
	\end{equation}
	Then $u$ cannot be a solution to problem ${\mathscr M}_A^c$.
\end{theorem}
\begin{proof}
	For every $u\in D^s_A(\R^3,\C)$ we have $|u|\in D^s(\R^3)$ and there holds 
	$$
	[u]_{s,A}^2-[|u|]_{s,0}^2=\int_{\R^6}\Upsilon^A_u(x,y)dxdy.
	$$
	Assume by contradiction	that $u$ solves ${\mathscr M}_A^c$. 
	Then, since $\|u\|_{L^{6/(3-2s)}}=1$, by Lemma \ref{livelli} and assumption 
	\eqref{Aass-00}, we conclude that ${\mathscr M}_0^c={\mathscr M}_A^c=[u]_{s,A}^2>[|u|]_{s,0}^2\geq {\mathscr M}_0^c,$
	a contradiction.
\end{proof}

%
\noindent
As a consequence we get (\ref{ii13}) of Theorem \ref{main-2}.

\begin{corollary}[Nonexistence of constant phase solutions]
Assume that 
$$
(x-y)\cdot A(x+y)\not \equiv k\pi,\quad\text{for some $k\in\N$ and on some $E\subset\R^6$ with ${\mathcal L}^6(E)>0$.}
$$
Then ${\mathscr M}_A^c$ does not admit solutions $u\in D^s_A(\R^3,\C)$ of the form $u(x)=e^{\i \vartheta} v(x)$
for some $\theta\in\R$ and $v\in D^s_A(\R^3,\R)$ of fixed sign.
\end{corollary}
\begin{proof}
	Assumption \eqref{Aass-00} is fulfilled, since
$$
\Upsilon_u^A(x,y)=2\big(1-\cos\big((x-y)\cdot A\big(\frac{x+y}{2}\big)\big)\big)v(x)v (y)>0,\quad 
\text{for a.e. $(x,y) \in E$}. 
$$		
Hence, the assertion follows from Theorem~\ref{nonex}.
\end{proof}

\bigskip
\medskip


\bigskip

\begin{thebibliography}{99}

\bibitem{A}
A. Applebaum, {\it L\'evy processes and Stochastic Calculus}, Cambridge Studies in Advanced Mathematics {\bf 116}, Cambridge University Press, Cambridge, 2009.

\bibitem{arioliSz}
G.\ Arioli, A.\ Szulkin,
{\it A semilinear Schr\"odinger equation in the presence of a magnetic field}, Arch. Ration. Mech. Anal. {\bf 170} (2003), 277--295.
	
\bibitem{AHS}
J. Avron, I. Herbst, B. Simon, {\it Schr\"odinger operators with magnetic fields. I. General interactions}, Duke Math. J. 45 (1978), 847--883.

\bibitem{BP}
L.\ Brasco, E.\ Parini, {\it The second eigenvalue of the fractional $p$-Laplacian}, 
Adv. Calc. Var. {\bf 9} (2016), 323--355.

\bibitem{bourg}
J. Bourgain, H. Brezis and P. Mironescu, 
{\it Another look at Sobolev spaces},
in \emph{Optimal Control and Partial Differential Equations. A Volume in Honor of Professor Alain Bensoussan's 60th Birthday}
(eds. J. L. Menaldi, E. Rofman and A. Sulem), IOS Press, Amsterdam, 2001, 439--455.

\bibitem{bourg2}
J. Bourgain, H. Brezis and P. Mironescu,
{\it Limiting embedding theorems for $W^{s,p}$ when $s \uparrow 1$ and applications},
{J. Anal. Math.} \textbf{87} (2002), 77--101.

\bibitem{classif}
W.\ Chen, C.\ Li,  B.\ Ou,  
{\it Classification of solutions for an integral equation},
Comm.\ Pure Appl.\ Math. {\bf 59} (2006), 330--343.

\bibitem{CS}
S.\ Cingolani, S.\ Secchi, {\it Semiclassical limit for nonlinear Schr\"odinger equations with electromagnetic fields}, J. Math. Anal. Appl. {\bf 275} (2002), 108--130.

\bibitem{CT}
R. Cont, P. Tankov, {\it Financial modeling with jump processes}, Chapman \& Hall/CRC Financial Mathematics Series, Chapman \& Hall/CRC, Boca Raton, FL, 2004.

\bibitem{costav}
A.\ Cotsiolis, N.K.\  Tavoularis,
{\it Best constants for Sobolev inequalities for higher order fractional derivatives},
J. Math. Anal. Appl. {\bf 295} (2004), 225--236.

\bibitem{DSS}
P.\ d'Avenia, G.\  Siciliano, M.\ Squassina,  
{\it On fractional Choquard equations},
Math. Models Methods Appl. Sci. {\bf 25} (2015), 1447--1476.

\bibitem{DCVS}
J.\ Di Cosmo, J.\ Van Schaftingen, {\it Semiclassical stationary states for nonlinear Schr\"odinger equations under a strong external magnetic field}, J. Differential Equations {\bf 259} (2015), 596--627.

\bibitem{DPV}
E.\ Di Nezza, G.\ Palatucci, E.\ Valdinoci,
{\it Hitchhiker's guide to the fractional Sobolev spaces},
Bull. Sci. Math. {\bf 136} (2012), 521--573.

\bibitem{DipValPal}
S.\ Dipierro, G.\ Palatucci, E.\ Valdinoci,
{\it Existence and symmetry results for a Schr\"odinger type problem involving the fractional Laplacian},
Le Matematiche {\bf 68} (2013), 201--216.

\bibitem{EL}
M.\ Esteban, P.L.\ Lions, 
{\it Stationary solutions of nonlinear Schr\"odinger equations with an external 
magnetic field}, Partial differential equations and the calculus of variations, Vol. I, 
401--449, Progr. Nonlinear Differential Equations Appl. {\bf 1}, Birkh\"auser Boston, Boston, MA, 1989.

\bibitem{franketal}
R.L.\ Frank, E.H.\ Lieb, R.\ Seiringer, 
{\it Hardy-Lieb-Thirring inequalities for fractional Schr\"odinger operators}, 
J. Amer. Math. Soc. {\bf 21} (2008), 925--950.

\bibitem{I89}
T.\ Ichinose, {\it Essential selfadjointness of the Weyl quantized relativistic Hamiltonian}, Ann. Inst. H. Poincar\'e Phys. Th\'eor. {\bf 51} (1989), 265--297.

\bibitem{I10}
T.\ Ichinose, {\it Magnetic relativistic Schr\"odinger operators and imaginary-time path integrals}, Mathematical physics, spectral theory and stochastic analysis, 247--297, Oper. Theory Adv. Appl. {\bf 232}, Birkh\"auser/Springer Basel AG, Basel, 2013.

\bibitem{IT}
T.\ Ichinose, H.\ Tamura, {\it Imaginary-time path integral for a relativistic spinless particle in an electromagnetic field}, Comm. Math. Phys. {\bf 105} (1986), 239--257. 

\bibitem{IMP}
V.\ Iftimie, M.\ M\u{a}ntoiu, R.\ Purice, {\it Magnetic pseudodifferential operators}, Publ. Res. Inst. Math. Sci. {\bf 43} (2007), 585--623.


\bibitem{K}
K. Kurata, {\it Existence and semi-classical limit of the least energy solution to a nonlinear Schr\"odinger equation with electromagnetic fields}, Nonlinear Anal. {\bf 41} (2000), 763--778.

%

\bibitem{Lask3}
N. Laskin, {\it Fractional Schr\"odinger equation}, 
Phys. Rev. E {\bf 66} (2002), 056108. 

\bibitem{LS}
E.H. Lieb, R. Seiringer, {\it The stability of matter in quantum mechanics},
Cambridge University Press, Cambridge, 2010.

\bibitem{Lions-c}
P.-L. Lions, {\it  Sym\'etrie and compacit\'e dans le espaces de Sobolev},
J. Funct. Anal {\bf 49} (1982), 315--334.

\bibitem{L1}
P.-L. Lions, {\it The concentration-compactness principle in the calculus of variation. The locally compact case I},
Ann. Inst. H. Poincar\'e Anal. Non Lin\'eaire {\bf 1} (1984), 109--145.

\bibitem{MK}
R. Metzler, J. Klafter, {\it The restaurant at the random walk: recent developments in the description of anomalous transport by fractional dynamics}, J. Phys. A {\bf 37} (2004), R161--R208.

%

\bibitem{ReedSimon}
M. Reed, B. Simon, {\it Methods of modern mathematical physics, I, Functional analysis}, Academic Press, Inc., New York, 1980.

\bibitem{S}
M. Squassina, {\it Soliton dynamics for the nonlinear Schr\"odinger equation with magnetic field}, Manuscripta Math. {\bf 130} (2009), 461--494.


\end{thebibliography}
\end{document}